\documentclass[final]{siamltex}

\usepackage{amsmath}
\usepackage{amsfonts} % Lukujoukot jne...
\usepackage{dsfont} % Lukujoukot jne...

\usepackage{graphicx}
\usepackage[oldcommands,linesnumbered,ruled]{algorithm2e}
\newtheorem{alg}[theorem]{Algorithm}
\newtheorem{defn}[theorem]{Definition}
\newtheorem{remark}{Remark}
\newtheorem{acknowledgement}{Acknowledgment}

\usepackage{hyperref}
\usepackage{caption}
\usepackage{subfig}

\usepackage{bbding}

\usepackage{tikz}
\usetikzlibrary{shapes,arrows}

\newcounter{minutes}
\divide\time by 60
\newcounter{hours}
\multiply\time by 60
\addtocounter{minutes}{-\time}
\newcommand{\C}{\mathbb{C}} % Kompleksiluvut
\newcommand{\symD}{\Omega} % Domain
\newcommand{\g}{\gamma} % kyr

\newcommand{\symQuad}{Q}
\newcommand{\symQuadC}{\tilde{Q}}
\newcommand{\symM}{\mathrm{M}}

\graphicspath{{PDF/}}  

\begin{document}

\title{Conjugate Function Method and Conformal \\ Mappings in Multiply Connected Domains}

\author{Harri Hakula\thanks{Aalto University, Institute of Mathematics,
P.O. Box 11100, FI-00076 Aalto,
FINLAND ({\tt harri.hakula@aalto.fi})} \and
Tri Quach\thanks{Aalto University, Institute of Mathematics,
P.O. Box 11100, FI-00076 Aalto,
FINLAND ({\tt tri.quach@aalto.fi})} \and
Antti Rasila\thanks{Aalto University, Institute of Mathematics,
P.O. Box 11100, FI-00076 Aalto,
FINLAND ({\tt antti.rasila@iki.fi})}}
%\fntext[fn1]{The author was supported by a grant (MA2012n21) from the Magnus Ehrnrooth Foundation.}

\maketitle

\begin{abstract}
The conjugate function method is an algorithm for numerical 
computation of conformal mappings for simply and doubly connected domains.
In this paper the conjugate function method is generalized for multiply connected domains.
The key challenge addressed here is the construction of the conjugate domain 
and the associated conjugate problem.
All variants of the method preserve the so-called reciprocal relation of the moduli.
An implementation of the algorithm, along with several examples and illustrations are given.
\end{abstract}

%\subjclass{Primary 30C30; Secondary 65E05, 31A15, 30C85}

% fill out if necessary or keep empty, acknowledgements go before bibliography
%\thanks{${}^\textrm{{\tiny\EightFlowerPetal}}$ The author was supported by a grant (MA2012n21) from the Magnus Ehrnrooth Foundation.}


%\linenumbers

\section{Introduction}
\noindent Conformal mappings play an important role in both theoretical 
complex analysis and in certain engineering applications, such as 
electrostatics, aerodynamics, and fluid mechanics. 
Existence of conformal mappings of simply connected 
domains onto the upper-half plane or the unit disk follows from 
the Riemann mapping theorem, and there are generalizations of this result 
for doubly and multiply connected domains \cite{ahl2}. However, constructing 
such mappings analytically is usually very difficult, and 
numerical methods are required.

There exists an extensive literature on numerical construction of conformal mappings 
for simply and doubly connected domains \cite{ps}. 
One popular method is based on the Schwarz-Christoffel formula \cite{dt}, and its implementation 
SC Toolbox is due to Driscoll \cite{driscoll96,dri}. 
SC Toolbox itself is based on earlier FORTRAN package by Trefethen \cite{trefethen}.
A new algorithm involving a finite element method and the 
harmonic conjugate function was presented by the authors in \cite{hqr}. 

While the study of numerical conformal mappings in multiply connected domains dates back to 1980's \cite{mayo,reichel}, recently
there has been significant interest towards the subject. DeLillo, Elcrat and Pfaltzgraff \cite{dep} were the first to give 
a Schwarz-Christoffel formula for unbounded multiply 
connected domains. Their method relies on the Schwarzian reflection principle. 
Crowdy \cite{crowdy1} was the first to derive a Schwarz-Christoffel formula for 
bounded multiply connected domains, which was based on the use of Schottky-Klein 
prime function. 
In a very recent paper \cite{sete} conformal maps from multiply connected domains onto lemniscatic domains have been discussed.
The natural extension of this result to unbounded multiply connected domains 
is given in \cite{crowdy2}. It should be noted that a MATLAB implementation 
of the Schottky-Klein prime function is freely available \cite{crowdy-green}, 
and the algorithm is described in \cite{crowdy-marshall}. A method involving 
the harmonic conjugate function is given in \cite{ldgy}, but the approach there 
differs from ours. 

\begin{figure}
  \centering
  \subfloat[$R$-type; Dirichlet on all boundaries.]{\includegraphics[width=0.45\textwidth]{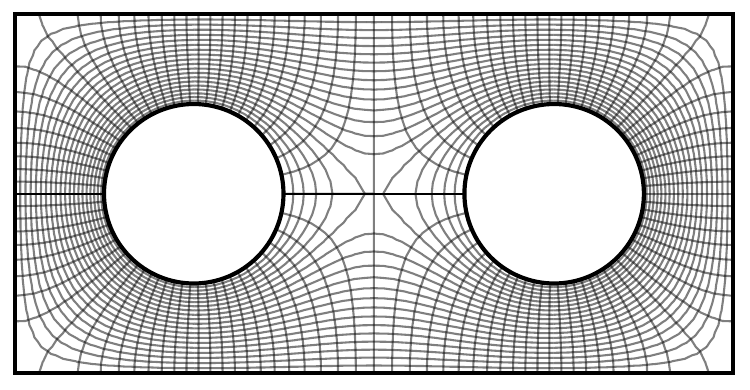}}\quad
  \subfloat[$Q$-type; Dirichlet on the left and right edges.]{\includegraphics[width=0.45\textwidth]{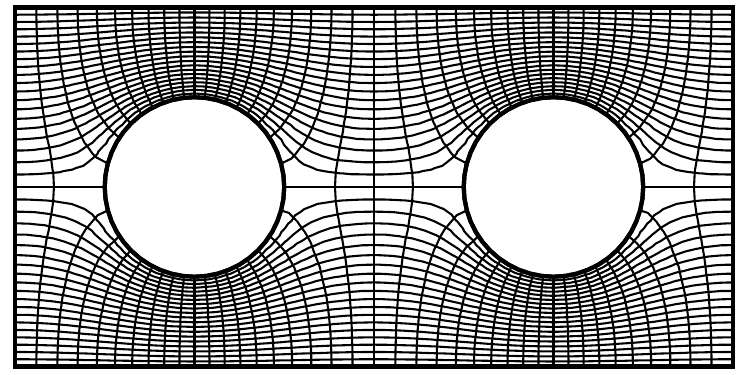}}
  \caption{Two Circles in Rectangle: Map.}\label{fig:introduction}
\end{figure}

The foundation of conjugate function methods for simply and doubly 
connected domains lies on properties of the (conformal) modulus, which
originates from the theory of quasiconformal mappings
\cite{ahlfors,lv,ps}. Here we extend the methods to multiply connected domains. In terms of partial differential equations, one has
to solve the Laplace equation $\Delta u = 0$, in $ \Omega$, 
with boundary conditions
\begin{equation}
		\mathds{1}_N\frac{\partial u}{\partial n} + \mathds{1}_D u = f(x,y), \text{ on } \partial\Omega,
\end{equation}
where the indicator functions refer to Neumann and Dirichlet boundary 
parts, respectively. Two configurations are of special interest: 
first, if only Dirichlet boundary conditions are set, e.g., 0 on the outer boundary,
and 1 on the interior boundary components, the problem is ring-like, and second, if only
two non-adjacent boundary segments have Dirichlet boundary conditions, the problem is 
quadrilateral-like; the configurations are referred to as types of $R$ and 
$Q$, respectively (See Figure~\ref{fig:introduction}).
In both cases the canonical domains are slit domains, first catalogued by 
Koebe \cite{Koebe}.

The main result of this paper is the generalization of the 
of the algorithm for simply and doubly connected domains described in
\cite{hqr} to multiply connected ones with different boundary conditions.
To our knowledge this is the first method for problems of type $R$.
More specifically, the fundamental new result
is the definition of the conjugate problem for multiply connected domains.
We show formally for type $R$ (Proposition \ref{prop: reci}) that this 
choice for the conjugate problem preserves the
important reciprocal relation \cite{lv} for the moduli 
$\symM(\symD)$ and $\symM(\tilde{\symD})$
of the original and the conjugate problem, respectively:
\[
	\symM(\symD) \symM(\tilde{\symD})=1.
\]
Similar result holds for type $Q$.

Our method is suitable for a very general class of
domains, allowing curved boundaries and even cusps. The
implementation of the algorithm is based on the $hp$-FEM described in
\cite{hrv1}, and in \cite{hrv2} it is generalized to cover unbounded
domains. In \cite{hrv3}, the method has been used to compute moduli
of domains with strong singularities.

The performance of the method has been evaluated by solving four
benchmark problems, two on computing resistances \cite{dek, trefethen2},
and two on capacities \cite{bsv}.
In each case the results agree with those obtained either with special-purpose methods or adaptive $h$-FEM.

In general, conformal mapping of multiply connected domains
is possible only if the the domain is a conformal image of a Denjoy-domain, i.e., 
a domain complement of which is a subset of the real line. 
It is well-known that this property holds for any $n$ times connected domain if $n$ is $1$, $2$, or $3$. 
In the method presented here, for instance in cases of type $R$ the saddle points of the potential function
of the original multiply connected problem are special, and it may be that the mapping is not conformal
exactly at the saddle point if the domain is not a Denjoy-domain. 
Thus, our method is conformal up to a finite set of points (\textit{exceptional points}). 

The rest of the paper is organized as follows: 
In Section~\ref{sec:preliminaries} the necessary concepts from function theory are introduced.
The new algorithms for multiply connected domains is described in Sections~\ref{sec:cfmmcd}
and \ref{sec:cfmmcdQ}, for types of $R$ and $Q$, respectively.
After the numerical implementation is briefly discussed, an extensive set of numerical experiments
is analyzed. As the final example of the paper we show how
canonical domains can be used for tracking evolving solutions, e.g., stress
fields, as the computational domain is perturbed.

\section{Preliminaries} \label{sec:preliminaries}

In this section we introduce concepts from function
theory and review the algorithm for simply or doubly connected domains.
For details and references we refer to \cite{hqr}.

\begin{defn} (Modulus of a Quadrilateral) \\
A Jordan domain $\symD$ in $\C$ with marked (positively ordered) points 
$z_1,z_2,z_3,z_4\in \partial \symD$ is called a {\it quadrilateral}, and denoted 
by $\symQuad = (\symD;z_1,z_2,z_3,z_4)$. Then there is a canonical conformal map of the quadrilateral $\symQuad$ onto a rectangle $R_d = (\symD';1+id,id,0,1)$, with the vertices 
corresponding, where the quantity $d$ defines the  {\it modulus of a quadrilateral}
$\symQuad$. We write
\[
\symM(\symQuad) = d.
\]
\end{defn}
Notice that the modulus $d$ is unique.
\begin{lemma} (Reciprocal Identity) \\
The following reciprocal
identity holds:
\begin{equation} \label{eqn: recip}
\symM(\symQuad)\, \symM(\symQuadC) =1,
\end{equation}
where $\symQuadC= (\symD; z_2,z_3,z_4, z_1)$ is called the
 {\it conjugate quadrilateral} of $\symQuad$.
\end{lemma}

\subsection{Dirichlet-Neumann Problem}
\label{sec:dnprob}
It is well known that one can express the modulus of a quadrilateral $Q$
in terms of the solution of the Dirichlet-Neumann mixed boundary value
problem.

Let $\symD$ be a domain in the complex plane whose boundary $\partial
\symD$ consists of a finite number of piecewise regular Jordan curves, so that at
every point, except possibly at finitely many points of the boundary, an exterior
normal is defined.  Let $\partial \symD =A \cup B$ where $A, B$ both are
unions of regular Jordan arcs such that $A \cap B$ is finite. Let
$\psi_A$, $\psi_B$ be real-valued continuous functions defined on $A,
B$, respectively. Find a function $u$ satisfying the following
conditions:
\begin{enumerate}
\item $u$ is continuous and differentiable in $\overline{\symD}$.
\item $u(t) = \psi_A(t),\qquad \textrm{for all } \, t \in A$.
\item If $\partial/\partial n$ denotes differentiation in
the direction of the exterior normal, then
\[
\frac{\partial}{\partial n} u(t)=\psi_B(t),\qquad \textrm{for all } \, t \in  B.
\]
\end{enumerate}
The problem associated with the conjugate quadrilateral $\symQuadC$ is called the {\it conjugate Dirichlet-Neumann problem}.

Let $\gamma_j, j=1,2,3,4$ be the arcs of $\partial \symD$ between $(z_1,
z_2)\,,$ $(z_2, z_3)\,,$ $(z_3, z_4)\,,$ $(z_4, z_1),$ respectively.
Suppose that $u$ is the (unique) harmonic solution of the
Dirichlet-Neumann problem with mixed boundary values of $u$ equal to $0$
on $\gamma_2$, equal to $1$ on $\gamma_4$, and $\partial u/\partial n =
0$ on $\gamma_1, \gamma_3$. Then:
\begin{equation} \label{qmod}
\symM(\symQuad)= \iint_\symD |\nabla  u|^2\,dx \, dy.
\end{equation}

Suppose that $\symQuad$ is a quadrilateral, and $u$ is the harmonic
solution of the Dirichlet-Neumann problem and let $v$ be a conjugate
harmonic function of $u$,
$v(\textrm{Re}\, z_3, \textrm{Im}\,z_3)~=~0$. 
Then $f = u + iv$ is an analytic function, and it maps
$\symD$ onto a rectangle $R_h$ such that the image of the points
$z_1,z_2,z_3,z_4$ are $1+id, id,0,1$, respectively. Furthermore by
Carath\'{e}odory's theorem, $f$ has a
continuous boundary extension which maps the boundary curves $\g_1,
\g_2, \g_3, \g_4$ onto the line segments $\g_1', \g_2', \g_3', \g_4'$.
\begin{lemma} \label{lemma: conj-hqr}
Let $\symQuad$ be a quadrilateral with modulus $d$, and let $u$ be the
harmonic solution of the Dirichlet-Neumann problem. Suppose that $v$ is
the harmonic conjugate function of $u$, with $v({\rm Re}\, z_3, {\rm
Im}\, z_3) = 0$. If  $\tilde{u}$ is the harmonic solution of the
Dirichlet-Neumann problem associated with the conjugate quadrilateral
$\symQuadC$, then $v = d\tilde{u}$.
\end{lemma}

\subsection{Ring Domains}
Let $E_0$ and $E_1$ be two disjoint and connected compact sets in the extended 
complex plane ${\C_\infty} = \C \cup \{\infty\}$. Then one of the set $E_0$ or $E_1$ is bounded and
without loss of generality we may assume that it is $E_0$. Then a set
$R={\C_\infty} \backslash (E_0 \cup E_1)$ is connected and is called a
{\it ring domain}. The {\it capacity} of $R$ is defined by
\[
\textrm{cap} R = \inf_u \iint_R |\nabla u|^2 \, dx \, dy,
\]
where the infimum is taken over all non-negative, piecewise
differentiable functions $u$ with compact support in $R\cup E_0$ such
that $u=1$ on $E_0$. Suppose that a function $u$ is defined on $R$ with
$1$ on $E_0$ and $0$ on $E_1$. Then if $u$ is harmonic, it is unique and
it minimizes the integral above. The conformal modulus of a ring domain
$R$ is defined by $\symM(R) = 2\pi / \textrm{cap} R$. The ring domain
$R$ can be mapped conformally onto the annulus $A_r$, where $r =
\symM(R)$.
\begin{figure}
\centering
\subfloat[Ring domain with boundary conditions. $\Gamma$ is one of the contours, i.e., equipotential curves of the solution $u_1$.]{\includegraphics[width=0.3\textwidth]{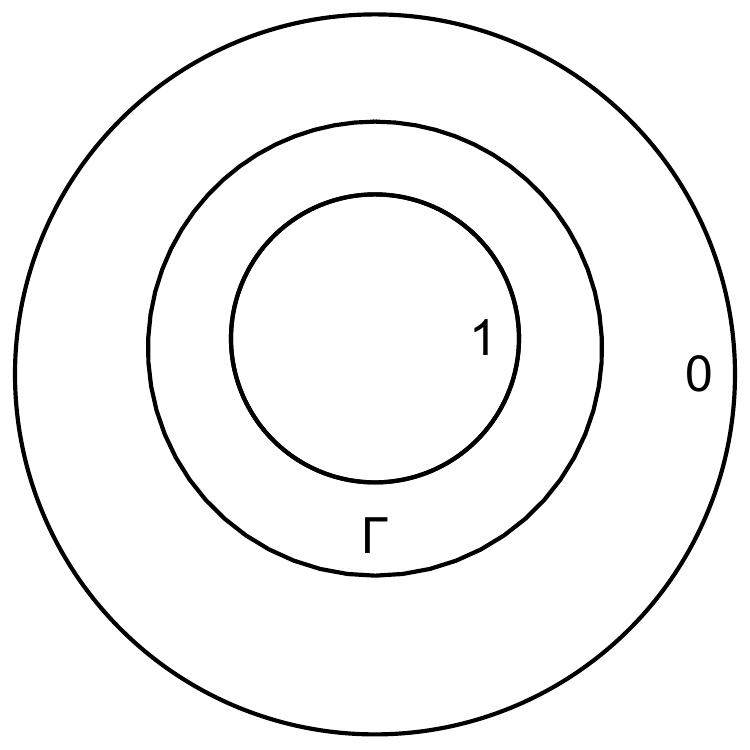}}\quad
\subfloat[Conjugate domain $\tilde{\Omega}$ with boundary conditions. Here the Dirichlet boundary conditions are taken to be 0 and $d = \int_{\Gamma} |\nabla u_1| ds$.]{\includegraphics[width=0.3\textwidth]{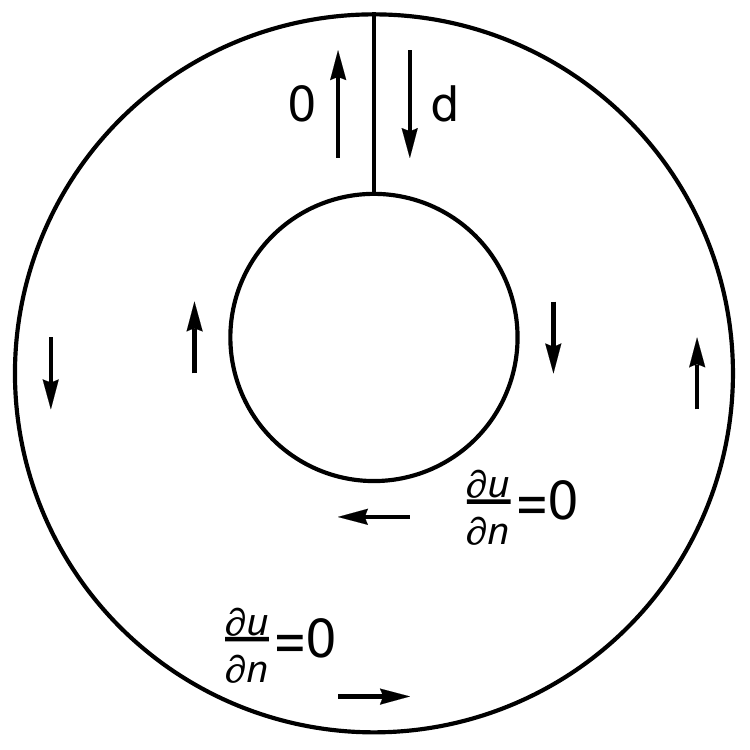}}\quad
\subfloat[Conformal map: Contour lines of $u_1$ and $u_2$.]{\includegraphics[width=0.3\textwidth]{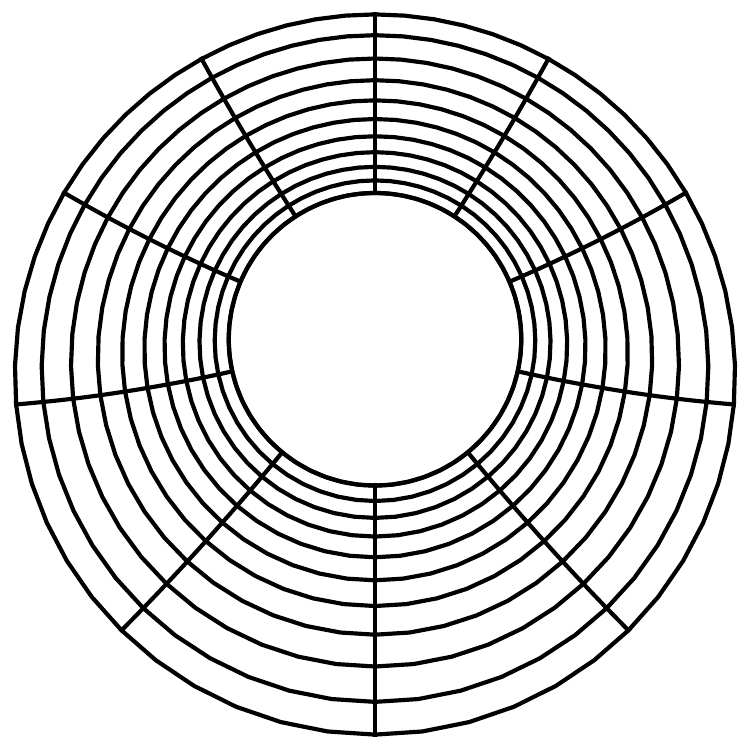}}
\caption{Introduction to the conjugate function method for ring domains.}\label{fig:fig}
\end{figure}

\subsection{Conjugate Function Method} \label{sec:cfm}
For simply connected domains the conjugate function method can be defined
in three steps.
\begin{alg}(Conjugate Function Method) \label{alg: hqr}
%\\
%\vspace*{-0.7cm}
\begin{enumerate}
\item Solve the Dirichlet-Neumann problem to obtain $u_1$ and compute the modulus $d$.
\item Solve the Dirichlet-Neumann problem associated with $\symQuadC$ to obtain $u_2$.
\item Then $f = u_1 + idu_2$ is the conformal mapping from $\symQuad$ onto $R_d$ such that 
  the vertices $(z_1,z_2,z_3,z_4)$ are mapped onto the corners $(1+id,id,0,1)$.
\end{enumerate}
\end{alg}
For ring domains the algorithm has to be modified, of course, and here the fundamental
step is the cutting of the domain along the path of steepest descent, which enables
us to return the problem to similar settings as for the simply connected case.
\begin{alg}(Conjugate Function Method for Ring Domains) \label{alg: hqrring}
%\\
%\vspace*{-0.7cm}
\begin{enumerate}
\item Solve the Dirichlet problem to obtain the potential function
$u$ and the modulus $\symM(R)$.
\item  Cut the ring domain through the steepest descent curve which
is given by the gradient of the potential function $u_1$ and obtain a quadrilateral where the Neumann condition is on the
steepest descent curve and the Dirichlet boundaries remain as before. 
\item Use the method for simply connected domains (Algorithm \ref{alg: hqr}).
\end{enumerate}
\end{alg}
Notice that the choice of the steepest descent curve is not unique due
to the implicit orthogonality condition.
In Figure~\ref{fig:fig} an example of the ring domain case is given.
The key observation is that  $d = \int_{\Gamma} |\nabla u_1|\, ds$, 
where $\Gamma$ is any of the contour lines of the solution $u_2$.
In  Figure~\ref{fig:fig}b the Dirichlet boundary conditions are set to be 0 and $d$,
instead of usual choice of 0 and 1. This choice does not have any effect for Figure~\ref{fig:fig}c
but is of paramount interest in the generalization of the algorithm.

\begin{defn}[Cut]\label{def:cut}
  A cut $\gamma$ is a curve in the domain $\Omega$, which introduces
  two boundary segments denoted by $\gamma^+$ and $\gamma^-$ to the 
  conjugate domain $\tilde{\Omega}$. Along the oriented boundary
 $\partial\tilde{\Omega}$, the segments  $\gamma^+$ and $\gamma^-$ are traversed
 in opposite directions.
\end{defn}

For the sake of discussion below let us define the conjugate problem directly.
The cut $\gamma$ (Definition~\ref{def:cut}) has its end points on $\partial E_0$
and  $\partial E_1$. One choice for the (oriented) boundary of conjugate domain $\tilde{\Omega}$
starting from the end point of $\gamma$ on  $\partial E_1$ is given by the set 
$\{\gamma^+,\partial E_0,\gamma^-,\partial E_1 \}$ as shown in  Figure~\ref{fig:fig}b.
The boundary conditions are set as $u = 0$ on $\gamma^+$, $u = d$ on $\gamma^-$, and
$\partial u_2 / \partial n = 0$ on $\partial E_j$, $j=1,2$.
\subsection{Canonical Domains} \label{sec:canonical}
The so called canonical domains play a
crucial role in the theory of quasiconformal mappings (cf. \cite{lv}).
These domains have a simple geometric structure. Let us consider a
conformal mapping $f \colon \symD \to D$, where $D$ is a canonical
domain, and $\symD$ is the domain of interest. The choice of the
canonical domain depends on the connectivity of the domain $\symD$, and
both domains $D$ and $\symD$ have the same connectivity. It should be
noted that in simply and doubly connected cases, domains can be
mapped conformally onto each other if and only if their moduli agree.
%which is given by the energy norm (see e.g. \cite{lv}). 
In this sense,
moduli divide domains into conformal equivalence classes. For simply
connected domains, natural choices for canonical domains are the unit
disk, the upper half-plane and a rectangle. In the case of doubly
connected domains an annulus is used as the canonical domain. For
$m$-connected domains, $m>2$, we have $3m-6$ different moduli, which
leads to various choices of canonical domains. These domains have been
studied in \cite{grunsky,nehari}. The generalization of Riemann mapping
theorem onto multiply connected domains is based on these moduli, see
\cite[Theorems 3.9.12, 3.9.14]{grunsky}. 

\section{Conjugate Function Method for Multiply Connected Domains of Type $R$}\label{sec:cfmmcd}
Let us first formally define the multiply connected domains of type $R$ and their capacities.
\begin{figure}[!ht]
\begin{center}
\begin{tikzpicture}[scale=.8,>=stealth]

\filldraw[dashed,rounded corners=8pt,thin,fill=black!7!white] (0.6,-0.6) .. controls (0.6,1.6) .. (1.6,0.6) coordinate[near start] (z_1) .. controls (3,3.6) .. (2,3) .. controls (-1.4,3.4) and (-2,3) ..  (-1,-1.2) -- cycle;

\draw[dashed,thick,white] (z_1) -- ++(-.6,0) coordinate (z_2);
\draw[thick,color=orange] (z_1) -- (z_2) coordinate [midway] (zz1);
\draw (zz1) node [anchor=north] {$\gamma_0^+$};
\draw[thick, shorten >=0.1cm, shorten <=0.1cm, <-] (zz1) .. controls (.6,1.3)  .. ++(1.1,-0.4) node [anchor=west] {$\gamma_0^-$};

\filldraw[thick,white,fill=white] (z_2) arc (0:360:.45cm);

\draw[dashed,rounded corners=8pt,thick] (0.6,-0.6) .. controls (0.6,1.6) .. (1.6,0.6) .. controls (3,3.6) .. (2,3) .. controls (-1.4,3.4) and (-2,3) ..  (-1,-1.2) -- cycle;

%\draw[thick] (z_2) (-.2,.2) coordinate(z_3);
\draw[thin] (z_2) arc (0:60:.45cm) coordinate (z_3);% node[anchor=north east] {$E_1$};
\draw[dashed,thick,white] (z_3) -- ++(.3,.7) coordinate (z_4);
\draw[dashed,thick,color=red] (z_3) -- (z_4) coordinate [midway] (zz2);
\draw (zz2) node [anchor=east] {$\gamma_1^+$};
\draw[thick, shorten >=0.1cm, shorten <=0.1cm, <-] (zz2) .. controls (.9,1.50)  .. ++(1.3,0.5) node [anchor=west] {$\gamma_1^-$};

\draw[thick,color=cyan] (z_2) arc (0:360:.45cm);
\filldraw[thick,white,fill=white] (z_4) arc (-120:240:.55cm);
\draw[thick,dotted,color=blue] (z_4) arc (-120:240:.55cm);

% f(z) arrow
\draw[->,thick,xshift=2.5cm,yshift=1cm] (0,0) -- (1,0) node [midway,anchor=south]{$\varphi$};
\draw[->,thick,xshift=7.7cm,yshift=1cm] (0,0) -- (1,0) node [midway,anchor=south]{scaling};
\draw[thin,xshift=7.7cm,yshift=1cm] (0,0) -- (1,0) node [midway,anchor=north]{exp $\circ$ rot};

\filldraw[white,thick,fill=black!7!white,xshift=5.5cm,yshift=1cm] (-1.5,-2.5) -- (1.5,-2.5) -- (1.5,2.5) -- (-1.5,2.5) -- cycle;

\draw[thick, dashed,xshift=5.5cm,yshift=1cm, color=red] (1.5,-1.2) coordinate (ww1) -- (0.8,-1.2) ;
\draw (ww1) node [anchor=south east] {$\varphi(\gamma_1^+)$};
\draw[thick, dashed,xshift=5.5cm,yshift=1cm, color=red] (1.5,1.2) coordinate (ww2) -- (0.8,1.2);
\draw (ww2) node [anchor=north east] {$\varphi(\gamma_1^-)$};
\draw[thick,dotted, xshift=5.5cm,yshift=1cm, color=blue] (1.5,-1.1823) -- (1.5,1.1823);
\draw[thick, xshift=5.5cm,yshift=1cm, color=cyan] (1.5,-2.5) -- (1.5,-1.1823) (1.5,1.1823) -- (1.5,2.5);

\draw[dashed,thick, xshift=5.5cm,yshift=1cm] (-1.5,2.5) -- (-1.5,-2.5);
%\draw[thick, xshift=5.5cm,yshift=1cm, color=cyan] (1.5,-2.5) -- (1.5,-1.7823) (1.5,1.8) -- (1.5,2.5);
\draw[thick, xshift=5.5cm,yshift=1cm, color=orange] (-1.5177,-2.5) -- (1.5177,-2.5) coordinate [midway] (ww3) (1.5177,2.5) -- (-1.5177,2.5) coordinate [midway] (ww4);
\draw (ww3) node [anchor=south] {$\varphi(\gamma_0^+)$};
\draw (ww4) node [anchor=north] {$\varphi(\gamma_0^-)$};

\filldraw[thick,shift={(11.7cm,1cm)},white,fill=black!7!white] (0,0) circle (2.3cm);
\filldraw[thick,shift={(11.7cm,1cm)},white,fill=white] (0,0) circle (.75cm);
\draw[thick,dashed,shift={(11.7cm,1cm)}] (.75,0.02) -- (2.3,0.02);

\draw[thick,dashed,shift={(11.7cm,1cm)},color=red] (1.25,-0.015) arc (0:-23:1.25);
\draw[thick,dashed,shift={(11.7cm,1cm)},color=red] (1.8,-0.015) arc (0:-23:1.8);
\draw[thick,dotted,shift={(11.7cm,1cm)},color=blue] (1.2675,-.015) -- (1.7825,-.015);
\draw[thick,shift={(11.7cm,1cm)},color=cyan] (.75,-.015) -- (1.2675,-0.015);
\draw[thick,shift={(11.7cm,1cm)},color=cyan] (1.7825,-.015) -- (2.3,-.015);

\draw[thick,shift={(11.7cm,1cm)},color=orange] (0,0) circle (2.3cm);
\draw[thick,shift={(11.7cm,1cm)}, color=orange] (0,0) circle (.75cm);

\end{tikzpicture}
\caption{Construction of the conformal mapping from the domain of interest onto a canonical domain. In the first part, we use Algorithm \ref{alg: hqrmcd}, which creates the orange cut $\gamma_0$ and the dashed red cut $\gamma_1$. In the algorithm these cuts are traversed twice, which lead to two separated line-segments $\varphi(\gamma_k^+)$ and $\varphi(\gamma_k^-)$ on the rectangle. The latter part consist of a rotation, a scaling, and finally mapping with the exponential function.} \label{fig: construction}
\end{center}
\end{figure}
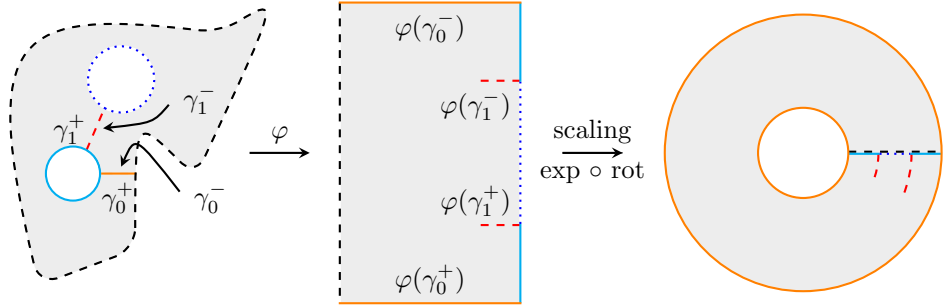
Let $m>2$ and $E_0, E_1, \ldots, E_m$ be disjoint and nondegenerate continua in the extended 
complex plane ${\C_\infty} = \C \cup \{\infty\}$. Suppose that $E_j$, $j = 1, 2, \ldots, m$ are bounded, then a set $\symD_{m+1} = \C_\infty \backslash \bigcup_{j=0}^{m} E_j$ is $(m+1)$-connected domain, and its (conformal) capacity is defined by
\[
\textrm{cap } \symD_{m+1} = \inf_{u} \iint_{\symD_{m+1}} |\nabla u|^2 \, dx \, dy,
\]
where the infimum is taken over all non-negative, piecewise
differentiable functions $u$ with compact support in
$\bigcup_{j=1}^{m} E_j \cup \symD_{m+1}$ such that $u=1$ on $E_j$, $j=
1, 2, \ldots, m$. Suppose that a function $u$ is defined on
$\symD_{m+1}$ with $1$ on $E_j$, $j=1,2, \ldots, m$ and $0$ on $E_0$.
Then if $u$ is harmonic, it is unique, and it minimizes the
integral above. The modulus of $\symD_{m+1}$ is defined by
$\symM(\symD_{m+1}) = 2\pi / \textrm{cap } \symD_{m+1}$.
%$\symM(\symD_{j+1}) = h = 2\pi / \textrm{cap } \symD_{j+1}$. 
If the degree of connectivity does not play an important role, the subscript will be omitted and we simply write $\symD$.

In contrast with the ring problem there is no immediate way to define a conjugate problem.
Indeed, it is clear that the conjugate domain \textit{cannot} be a quadrilateral in the sense of
definitions above. However, there exists a contour line $\Gamma_0$ such that it encloses the
set $E_j,\ j=1,2,\ldots$, and 
\begin{equation} \label{eqn: constant-d}
d = \symM(\symD_{m+1}) =  \int_{\Gamma_0} |\nabla u| \, ds.
\end{equation}

Thus, there is an analogue for the cutting of the domain along the curve of steepest descent.
It can be assumed without 
any loss generality, that the cut $\gamma_0$ (and the Dirichlet conditions) is between $E_0$ and $E_1$.
Then the immediate question is how to cut the domain further between $E_j,\ j=1,2,\ldots, m$, in 
such a fashion that the conjugate domain is simply connected, and set the boundary conditions
so that the Cauchy-Riemann equations are satisfied? 

There is one additional property of 
the solution $u$ that we can utilize. Namely, for every  $E_j,\ j=1,2,\ldots$, there exists an
enclosing contour line $\Gamma_j$.
The capacity has a natural decomposition
\begin{equation} \label{eq:decomposition}
  d = \sum_j \hat{d}_j,\quad \hat{d}_j =  \int_{\Gamma_j} |\nabla u|\, ds = \sum_k d_k = \sum_k\int_{\Gamma_{j,k}} |\nabla u|\, ds, 
\end{equation}
where $\Gamma_{j,k}$ denotes a segment from discretization of the contour line $\Gamma_j = \cup_k \Gamma_{j,k}$.

%The algorithm for constructing the conjugate problem is derived in three steps:
%identification of the saddle points of $u$; cutting of the domain $\Omega$; 
%determining the Dirichlet conditions over the cuts.
\subsection{Saddle Points} The saddle points of the solution $u$ are of special interest.
Notice that for simply and doubly connected domains they do not exist, thus any generalization of the Algorithm~\ref{alg: hqr}
must address them specifically. First, there are two steepest-descent curves emanating from every saddle point.
This means that in the conformal mapping of the domain slits will emerge since the potential at the saddle
point must be less than 1. Second, analogously there are two steepest/ascent curves reaching some boundary points
$z_i$, $z_j$, at boundaries $\partial E_i$, $\partial E_j$, respectively. In addition, we say that $E_i$ and $E_j$ are conformally visible to each other.

\begin{remark}
  For symmetric configurations there may be more than two steepest-descent and steepest-ascent curves
  at the saddle point.
\end{remark}
\subsection{Cutting Process} The orthogonality requirement implies that the curve formed by joining
two curves of steepest descent from $E_i$ and $E_j$ meeting at the saddle point must
be a contour line of the conjugate solution, that is, an equipotential curve.
It follows that as in the doubly connected case, both boundary segments induced by a cut
have a different Dirichlet condition.
Therefore the cutting process can be outlined as follows:
\begin{alg}(Cutting Process) \label{alg:cut}
%\\
%\vspace*{-0.7cm}
\begin{enumerate}
\item Identify the saddle points $s_k, \ k=1,2,\ldots$.
\item Join the two curves of steepest descent from $\partial E_i$ and $\partial E_j$ meeting at the point $s_k$ into
  cut $\gamma_m,\ m \geq 1$.
\item Starting from the first cut, form an oriented boundary of a simply connected domain by alternately 
  traversing cuts $\gamma_m$ and segments of $\partial E_j$ induced by the cuts. Once the boundary is completed, every 
  cut has been traversed twice (in opposite directions) and every $\partial E_j$ has been traversed once.
\end{enumerate}
\end{alg}
In Figure~\ref{fig: cutting-jump} two configurations are shown.
\begin{remark}
  The symmetric case is covered if we allow for overlapping or partially overlapping cuts.
\end{remark}  
\subsection{Dirichlet Conditions Over Cuts}
Once the domain $\Omega$ has been cut and the oriented boundary of the conjugate domain $\tilde{\Omega}$
has been set up it remains to set the Dirichlet conditions over the cuts. Given that the first
cut leads to boundary conditions of 0 and $d$, it is sufficient to simply trace the oriented
boundary of $\tilde{\Omega}$ and maintain the cumulative sum of jumps in modules computed over
the segments $\Gamma_{j,k}$ connecting two consecutive cuts. Referring to Figure~\ref{fig: cutting-jump}
notice that the identity (\ref{eq:decomposition}) hold over the segments $\Gamma_{j,k}$.
\begin{alg}(Dirichlet Conditions Over Cuts) \label{alg:jumps}
%\\
%\vspace*{-0.7cm}
\begin{enumerate}
\item Set the Dirichlet boundary conditions of the boundary conditions induced by the first cut to 0 and $d$.
\item Trace the boundary starting from the zero boundary and update the cumulative sum of
  \[
  d_m =  \int_{\Gamma_{j,k}} |\nabla u| \, ds,
\]
where the $\Gamma_{j,k}$ are included in the order given by the boundary orientation.
\item At every cut set the Dirichlet condition to the cumulative sum reached at that point.
\end{enumerate}
\end{alg}

\subsection{Reciprocal Identity}

Suppose that $u_1$ is the (unique) harmonic solution of the Dirichlet-Neumann problem given in the beginning of Section \ref{sec:cfmmcd}.
%with mixed boundary values of $u$ equal to $0$ on $\partial E_0$, equal to $1$ on $\partial E_j$, $u=u_0$ on $\gamma_j$, $j=2,3,\ldots, m$, and $\partial u/\partial n = 0$ on $\gamma_1$.
%$\partial u/\partial n = 0$ on $\gamma_j$, $j=1,2,\ldots, m$. 
Let $u_2$ be a conjugate harmonic function of $u_1$ such that $u_2(\textrm{Re}\, \tilde{z}, \textrm{Im}\, \tilde{z}) = 0$, where $\tilde{z}$ is the intersection point of $E_0$ and $\gamma_0^+$.

Then $\varphi = u_1 + iu_2$ is an analytic function, and it maps $\symD$ onto a rectangle $R_d = \{ z\in \C : 0< \mathrm{Re}\,z <1, \, 0< \mathrm{Im}\,z <d \}$ minus $n-2$ line-segments, parallel to real axis, between points $(u_1(\tilde{z}_j),d_j)$ and $(1,d_j)$, where $\tilde{z}_j$ is the saddle point of the corresponding $j$th jump. In the process we have total of $n$ jumps. See Figure \ref{fig: construction} for an illustration of a triply connected example.

%\begin{remark}
%The Neumann conditions on $\gamma_j$, $j=1,2,\ldots, m$ and the fact on cutting curves on our original Laplace problem induces that Neumann conditions are exactly the same as putting Dirichlet conditions $u_{|\gamma_j}$ as in solution on Laplace problem.
%\end{remark}

Let $u_2$ be the harmonic solution satisfying following boundary values $u_2$ equal to $0$ on $\gamma_0^+$ and equal to $1$ on $\gamma_0^-$, Neumann conditions $\partial u_2/\partial n = 0$  on $\partial E_j$, $j=0,1,\ldots, m$. For the cutting curves $\gamma_j$, $j=1,2,\ldots, m$, we have Dirichlet condition and the value is the cumulative sum $\sum_{j=0}^m d_j $.
On the $n$th jump, we have on the corresponding cutting curve $\gamma_j$
\[
u_2= \frac{\sum_{j=0}^n d_j}{d}, 
\]
where $d_j$ are given by \eqref{eq:decomposition}. %This is the problem on $\tilde{\symD}$, which is also called a {\it conjugate problem} of $\symD$. 
Note that, if $\Gamma_0$ is an equipotential curve from $\gamma_0^+$ to $\gamma_0^-$, then we have 
\[
\symM(\tilde{\symD}) = \int_{\Gamma_0} |\nabla u_2| \, ds = \frac{1}{d}.
\]
Thus we have a following proposition, which has the same nature as the reciprocal identity given in \cite{hrv1}.
\begin{proposition}[Reciprocal identity] \label{prop: reci}
Suppose $u_1$ and $u_2$ are the solutions to problems on $\symD$ and $\tilde{\symD}$, respectively. If $\symM(\symD)$ denotes the integral of the absolute value of gradient of $u_2$ over the equipotential curve from $\gamma_0^+$ to $\gamma_0^-$, and $\symM(\tilde{\symD})$ denotes the same integral for $u_2$, then we have a normalized reciprocal identity
\begin{equation} \label{eqn: reci}
\symM(\symD) \symM(\tilde{\symD}) = 1,
\end{equation}
\end{proposition}
This reciprocal identity can be used in measuring the relative error of conformal mapping. It should be noted, that the mapping depends on $3m-6$ parameters, moduli. Thus, theoretically it is possible to have an incorrect result for some of the moduli such that the reciprocal identity holds. However, probability of consistently having incorrect moduli for significant applications is extremely low.

%For the so called {\it conjugate problem of} $\symD$, we find harmonic solution $\tilde{u}$ such that $\tilde{u}$ equal to $0$ on $\gamma_1^+$ and equal to $1$ on $\gamma_1^-$, and $\partial \tilde{u}/\partial n = 0$  on $\partial E_1, \partial E_j, \gamma_j, j=2,3,\ldots, m$.

%By reversing the Dirichlet and Neumann boundaries, we find harmonic solution $\tilde{u}$ such that $\tilde{u}$ equal to $0$ on $\gamma_1^+$ and equal to $1$ on $\gamma_1^-, \gamma_j$, $j=2,3,\ldots, m$, and  $\partial \tilde{u}/\partial n = 0$ on $\partial E_j$, $j=1,2,\ldots, m$. This problem is called {\it conjugate problem of} $\symD$.

%Conformal modulus connected to Dirichlet-Neumann problem of $u$ is defined by
%\[
%\symM(\symD_u) = \iint_{\symD} |\nabla u_0|^2 \, dx \, dy = h.
%\]
%\begin{defn} (Reciprocal Identity) \\
%It is clear by the geometry
%\cite[p. 15]{lv} or \cite[pp. 53-54]{ps} that the following reciprocal
%identity holds:
%\begin{equation} \label{eqn: recip}
%\symM(\symD_u)\, \symM(\symD_{\tilde{u}}) =1.
%\end{equation}
%\end{defn}

\begin{lemma} \label{lemma: conj-hqr2}
Let $\symD$ be a multiply connected domain and let $u$ be the harmonic solution of the Dirichlet-Neumann problem. Suppose that $v$ is the harmonic conjugate function of $u$ such that $v(\textrm{Re}\, \tilde{z}, \textrm{Im}\, \tilde{z}) = 0$, where $\tilde{z}$ is the intersection point of $E_0$ and $\gamma_0^+$, and $d$ is a real constant given by \eqref{eqn: constant-d}. If $\tilde{u}$ is the harmonic solution of the Dirichlet-Neumann problem associated with the conjugate problem of $\symD$, then $v = d\tilde{u}$.
\end{lemma}

\begin{proof}
It is clear that $v, \tilde{u}$ are harmonic. By Cauchy-Riemann equations, we have $\langle \nabla u, \nabla v \rangle = 0$. We may assume that the gradient of $u$ does not vanish on $\partial E_j, j=0,1, \ldots, m$. Then on $\partial E_0$, we have $n = -\nabla u / |\nabla u|$, where $n$ denotes the exterior normal of the boundary. Likewise, we have $n = \nabla u / |\nabla u|$ on $\partial E_j, j=1,2,\ldots,m$. Therefore
\[
\frac{\partial v}{\partial n} = \langle \nabla v, n \rangle = \pm\frac{1}{|\nabla u|} \, \langle \nabla v, \nabla u \rangle = 0.
\]
On the cutting curves, we have from Cauchy-Riemann equations $|\nabla u| = |\nabla v|$, and from the jumping between cutting curves that $d = \sum_{j=0}^n d_j$. These results together imply that on the $n$th jump, we have on the corresponding cutting curve $\gamma_k$
\[
v= \sum_{j=0}^n d_j.
\]
Then by the uniqueness theorem for harmonic functions \cite[p. 166]{ahl2}, we conclude that $v = d\tilde{u}$.

Lastly, the proof of univalency of $\varphi= u + iv$ follows from the proof of univalency of $f$ in \cite[Lemma 2.3]{hqr}.
\end{proof}
%Lastly the cutting curves $\gamma_j, j=2,3,\ldots,m$ are given by the gradient of $u_0$. By Cauchy-Riemann equations, we obtain $|\nabla v| = |\nabla u_0|$. Thus, we have $\partial v/\partial n = |\nabla u_0|$ on $\gamma_j, j=2,3,\ldots,m$. Therefore, $v$ and $\tilde{u}$ satisfy the same boundary conditions on $\partial E_j, j=0,1,\ldots,m$, and on $\gamma_j, j=2,3,\ldots,m$.

%By the definition of $\tilde{u}$, we get
%\[
%\frac{\partial \tilde{v}}{\partial n}  = \frac{\partial \tilde{u}}{\partial n} = 0,
%\]
%on $\partial E_j, j=0,1,\ldots,m$, and on the cutting curves $\gamma_j, j=2,3,\ldots,m$. 
%Therefore $v$ and $\tilde{u}$ satisfy the same boundary conditions on $\partial E_j, j=0,1,\ldots,m$, and on $\gamma_j, j=2,3,\ldots,m$.
%
%On Dirichlet boundaries, by definition of $d$ and Cauchy-Riemann equations, we have
%\[
%d = \int_{\Gamma^+} |\nabla v| \, ds = \int_{\Gamma^+} |\nabla u| \, ds = - \int_{\Gamma^-} |\nabla v| \, ds = -\int_{\Gamma^-} |\nabla u| \, ds,
%\]
%hence $v$ and $\tilde{u}$ obtain the same constant on $\gamma_1$.
%Then by the uniqueness theorem for harmonic functions \cite[p. 166]{ahl2}, we conclude that $v = \tilde{u}$.
%
%The proof of univalency of $\varphi_1= u + iv$ follows the proof of univalency of $f$ in \cite[Lemma 2.3]{hqr}.
%\end{proof}

\subsection{Outline of the Algorithm}
For convenience we use $\{\gamma\}$ and $\{\partial E\}$ to denote the sets of all cuts and boundaries, respectively.
\begin{alg}(Conjugate Function Method for Multiply Connected Domains of type $R$) \label{alg: hqrmcd}
%\\
%\vspace*{-0.7cm}
\begin{enumerate}
\item Solve the  Dirichlet problem to obtain the potential function
$u_1$ and the modulus $d = \symM(\symD)$.
\item Choose one path of steepest descent reaching the outer boundary $E_0$, $\gamma_0$.
\item Identify the saddle points $s_m$.
\item For every saddle point: Find paths $\gamma_k$, $k>1$, joining two conformally visible boundaries $\partial E_i$ and $\partial E_j$ by  
  finding the paths of steepest descent meeting at the point $s_m$.
\item For every $E_i$: Choose a corresponding contour $\Gamma_i$, compute its subdivision
  $\Gamma_{i,k}$ induced by the paths $\{\gamma\}$, and the corresponding jumps
  $d_k = \int_{\Gamma_{i,k}} |\nabla u| ds$. 
\item Construct the conjugate domain $\tilde{\symD}$ by forming an oriented boundary using paths $\{\gamma\}$ and $\{\partial E\}$.
\item Set the boundary conditions along paths $\{\gamma\}$ by accumulating jumps in the order
  of traversal.
\item Solve the Dirichlet-Neumann problem on $\tilde{\symD}$ for $u_2$.
\item Construct the conformal mapping $\varphi = u_1 + idu_2$.
\end{enumerate}
\end{alg}

\begin{figure}
  \centering
  \subfloat[Non-symmetric case: Two saddle points; Five jumps.]{
\begin{tikzpicture}[scale=1.2,>=stealth]

\draw[rounded corners=8pt,thick,fill=black!7!white] (-2.2,-2) -- (2.2,-2) -- (2.2,2.123) -- (-2.2,2.132) -- cycle; 

\draw[thick] (1,0.932) -- node[midway, anchor=north west] {$\gamma_2$} (0,-0.8) -- (-1,0.932) node[midway, anchor=north east] {$\gamma_1$};
\draw[thick] (0,-2) -- (0,-1.3) node[midway, anchor=east] {$\gamma_0$};

\filldraw[thick,fill=white] (0,-0.8) circle (.5cm) node [anchor=base] {$E_1$};
\filldraw[thick,fill=white] (1,0.932) circle (.5cm) node [anchor=base] {$E_3$};
\filldraw[thick,fill=white] (-1,0.932) circle (.5cm) node [anchor=base] {$E_2$};

%\draw[->,thick] (0,-0.8) ++(0,-0.6) ++ (-0.2183,0) arc (250:140:0.6cm);
\draw[->,thick] (0,-0.8) ++ (-0.1042,-0.5909) arc (260:130:0.6cm);
\draw[->,thick] (-1,0.932) ++(0.2052,-0.5638) arc (290:-50:0.6cm);
\draw[->,thick] (0,-0.8) ++(-0.2052,0.5638) arc (110:70:0.6cm);
\draw[<-,thick] (1,0.932) ++(-0.2052,-0.5638) arc (-110:230:0.6cm);
\draw[<-,thick] (0,-0.8) ++ (0.1042,-0.5909) arc (-80:50:0.6cm);

\draw (-0.5,-0.9) node [anchor=east] {$\Gamma_{1,1}$};
\draw (-1,0.932) ++(-0.2,0.45) node [anchor=south east] {$\Gamma_2$};
\draw (0,-0.2) node [anchor=south] {$\Gamma_{1,2}$};
\draw (1,0.932) ++(0.25,0.45) node [anchor=south west] {$\Gamma_3$};
\draw (0.55,-0.9) node [anchor=west] {$\Gamma_{1,3}$};
\end{tikzpicture}
}\quad
  \subfloat[Symmetric case: One saddle point; Four jumps.]{
\begin{tikzpicture}[scale=1.2,>=stealth]

%%%%%%%%%%%%%%%%
% SYMMETRINEN TAPAUS  %
%%%%%%%%%%%%%%%%
%\draw[rounded corners=8pt,thick,fill=black!7!white] (-2.2,-2) -- (2.2,-2) -- (2.2,2.123) -- (-2.2,2.132) -- cycle; 
\filldraw[thick,shift={(6cm,-0.2932cm)},fill=black!7!white] (0,0.3547) circle (2.2cm);
%\filldraw[thick,shift={(0cm,-0.2932cm)},fill=black!7!white] (0,0.3547) circle (2.2cm);

\draw[thick,shift={(6cm,-0.2932cm)}] (0,-0.8) -- (0,0.3547);
\draw[thick,shift={(6cm,-0.2932cm)}] (1,0.932) -- (0,0.3547);
\draw[thick,shift={(6cm,-0.2932cm)}] (-1,0.932) -- (0,0.3547);

%\draw[thick,shift={(6cm,-0.2932cm)}] (1,0.932) -- node[midway, anchor=north west] {$\gamma_3$} (0,-0.8) -- (-1,0.932) node[midway, anchor=north east] {$\gamma_2$};
\draw[thick,shift={(6cm,-0.2932cm)}] (0,-1.8453) -- (0,-1.3) node[midway, anchor=east] {$\gamma_0$};

\filldraw[thick,shift={(6cm,-0.2932cm)},fill=white] (0,-0.8) circle (.5cm);
\filldraw[thick,shift={(6cm,-0.2932cm)},fill=white] (1,0.932) circle (.5cm);
\filldraw[thick,shift={(6cm,-0.2932cm)},fill=white] (-1,0.932) circle (.5cm);

%%\draw[->,thick] (0,-0.8) ++(0,-0.6) ++ (-0.2183,0) arc (250:140:0.6cm);
\draw[->,thick,shift={(6cm,-0.2932cm)}] (0,-0.8) ++ (-0.1042,-0.5909) arc (260:100:0.6cm);
\draw[->,thick,shift={(6cm,-0.2932cm)}] (-1,0.932) ++(0.4596,-0.3857) arc (320:-20:0.6cm);
\draw[<-,thick,shift={(6cm,-0.2932cm)}] (1,0.932) ++(-0.4596,-0.3857) arc (-140:200:0.6cm);
\draw[<-,thick,shift={(6cm,-0.2932cm)}] (0,-0.8) ++ (0.1042,-0.5909) arc (-80:80:0.6cm);
%
%
%\draw[<-,thick,shift={(6cm,-0.2932cm)}] (0,-0.8) ++ (0.1042,-0.5706) arc (-80:50:0.6cm);
%
\draw[shift={(6cm,-0.2932cm)}] (-0.5,-0.9) node [anchor=east] {$\Gamma_{1,1}$};
\draw[shift={(6cm,-0.2932cm)}] (-1,0.932) ++(0,0.55) node [anchor=south] {$\Gamma_2$};
%\draw[shift={(6cm,-0.2932cm)}] (0,-0.2) node [anchor=south] {$\Gamma_3$};
\draw[shift={(6cm,-0.2932cm)}] (1,0.932) ++(0,0.55) node [anchor=south] {$\Gamma_3$};
\draw[shift={(6cm,-0.2932cm)}] (0.55,-0.9) node [anchor=west] {$\Gamma_{1,2}$};

\end{tikzpicture}
}
\caption{Examples of non-symmetric and symmetric domains with cuts $\gamma$ and decomposition of jumping curves $\Gamma_{i,k}$.} \label{fig: cutting-jump}

\end{figure}
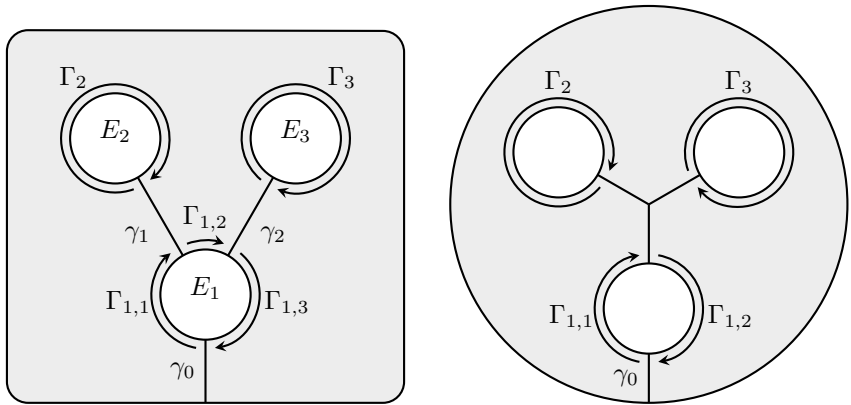

\subsection{Moduli and Degree of Freedom}
For $m+1$ connected domains, we have $3m-3$ different moduli, degrees of freedom. In general, we have $2m-1$ jumps, and $m-1$ saddle points. This sums up to $3m-2$. However the cut $\gamma_0$ can be chosen so that the first and the last jumps, $d_1$ and $d_{2m-1}$, respectively, are equal. Thus the number of degrees of freedom is reduced by one, and we obtain $3m-3$.

\section{Conjugate Function Method for Multiply Connected Domains of Type $Q$}
\label{sec:cfmmcdQ}

Let us next focus on the quadrilateral-like case, i.e., type $Q$. Conceptually the
construction is much simpler than that of type $R$. 
Let the exterior boundary $\partial E_0$
be composed of four arcs in the sense of Section~\ref{sec:dnprob} above, and
the interior boundaries $\partial E_j$, $j=1,\ldots,m$, have Neumann boundary conditions 
$\partial u/\partial n = 0$. Intuitively it is clear that the definition of the
conjugate problem has to involve a Dirichlet-Neumann map, and that there is
no need for any cutting process. Once the potentials over $\partial E_j$, $j=1,\ldots,m$,
have been defined for the conjugate problem, the reciprocal identity
follows immediately.

\subsection{Dirichlet Conditions Over Interior Boundaries}
\label{sec:dirichletQ}

Let us consider the configuration of Figure~\ref{fig:dirichletQ}. In the initial problem
the Dirichlet boundary conditions are $u=1$ and $u=0$ on left and right hand edges, respectively.
On every interior boundary $\partial E_j$, $j=1,\ldots,m$, there are exactly two points with
unique potentials that correspond to local maxima and minima, Figure~\ref{fig:dirichletQa}.
Let us consider $\partial E_j$ and denote the point with maximum potential $x$.
Point $x$ is connected with a point $s$ on either one of the Dirichlet boundaries
via a curve of steepest ascent $\rho$, Figure~\ref{fig:dirichletQb}. 
In the conjugate problem, the Dirichlet boundaries become Neumann ones.
Along the Neumann edges 
the solution will be linear and have all values in the interval $[0,1]$.
Thus, the potential at the point $s$, and by construction at $x$ since $\rho$
is an equipotential curve in the conjugate problem,
can be found using simple interpolation.
The same procedure can be applied to the point of local minimum on $\partial E_j$.
The resulting map is given in Figure~\ref{fig:dirichletQc}.

\begin{figure}
  \centering
  \subfloat[Local maxima and minima of the interior boundaries; Equipotential curves.]{\label{fig:dirichletQa}\includegraphics[width=0.3\textwidth]{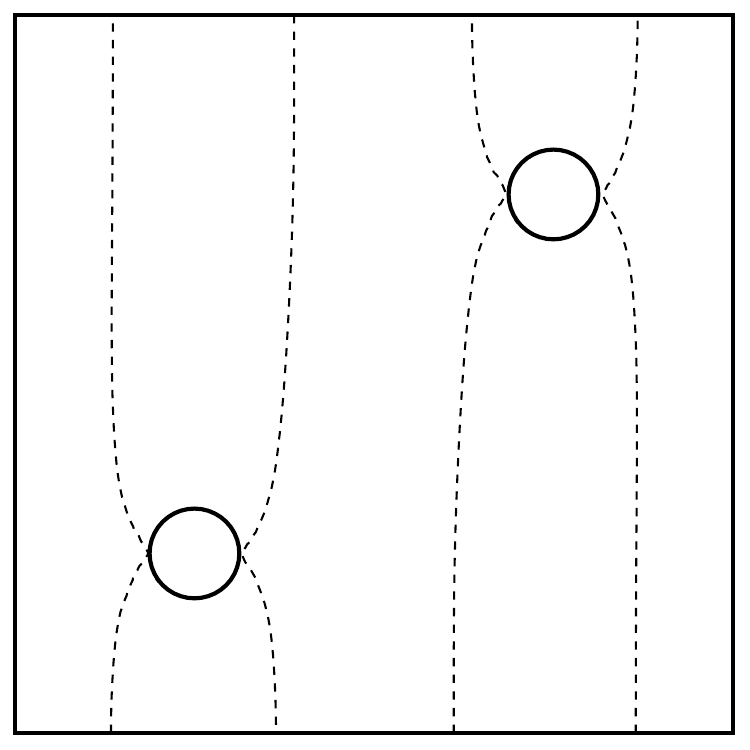}}\quad
  \subfloat[Geometric setting of potentials; Curves of steepest descent and ascent $\rho_m$.]{\label{fig:dirichletQb}\includegraphics[width=0.3\textwidth]{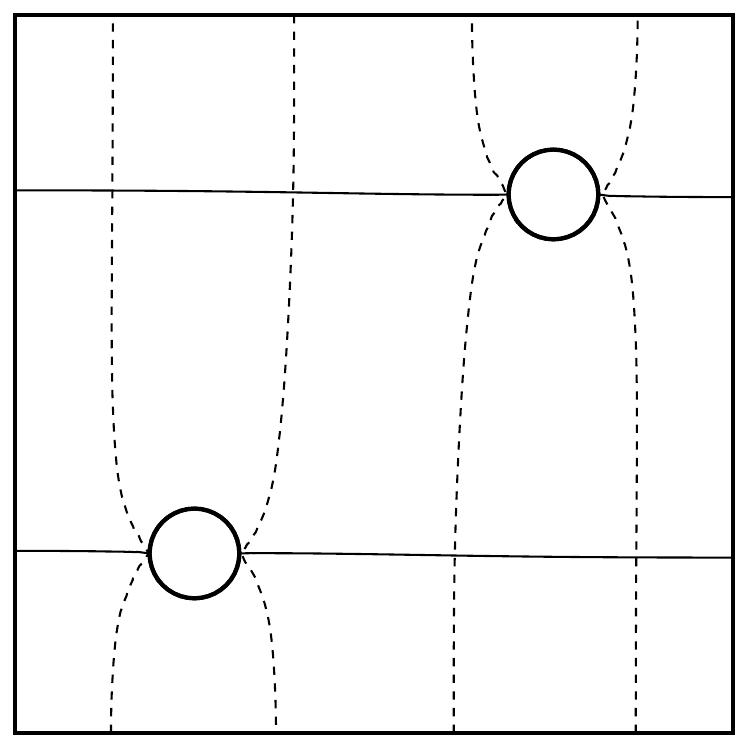}}\quad
  \subfloat[Map.]{\label{fig:dirichletQc}\includegraphics[width=0.3\textwidth]{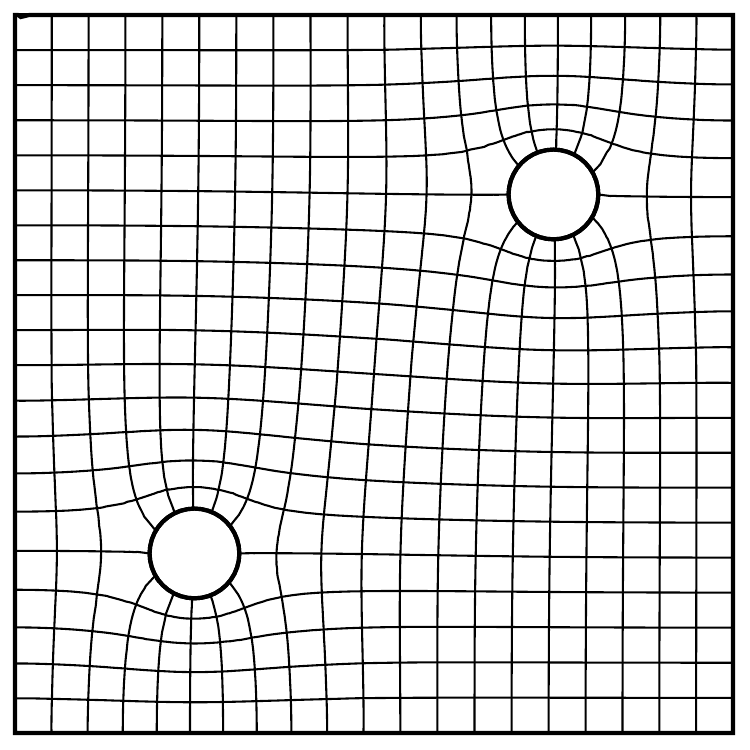}}
  \caption{$Q$-type: Dirichlet conditions for the conjugate problem: Initially on the left hand edge $u=1$ and right hand edge $u=0$.}\label{fig:qsteps}
\label{fig:dirichletQ}
\end{figure}

\subsection{Outline of the Algorithm}
Let us assume that in the initial Dirichlet-Neumann problem, along the boundary segment $\gamma_1$
the Dirichlet boundary condition is $u=1$. 
\begin{alg}(Conjugate Function Method for Multiply Connected Domains of type $Q$) \label{alg: hqrmcdq}
\begin{enumerate}
\item Solve the  Dirichlet-Neumann problem to obtain the potential function
$u_1$ and the modulus $d = \symM(\symD)$.
\item Locate the local maxima and minima on the interior boundaries $\partial E_j$, $j=1,\ldots,m$.
\item For every local maximum $x_m$: Find paths of steepest ascent $\rho_m$, $m>1$, 
connecting $x_m$ on $\partial E_i$ with the point $s_m$ on $\gamma_1$.
\item Interpolate the potential on $s_m$ on $\gamma_1$ when $\gamma_1$ is a Neumann edge.
\item Construct the conjugate domain $\tilde{\symD}$ by
performing the Dirichlet-Neumann map on $\partial E_0$ and setting the Dirichlet boundary
conditions on $\partial E_j$, $j=1,\ldots,m$, to values obtained in the previous step.
\item Solve the Dirichlet-Neumann problem on $\tilde{\symD}$ for $u_2$.
\item Construct the conformal mapping $\varphi = u_1 + idu_2$.
\end{enumerate}
\end{alg}

\section{Numerical Implementation of the Algorithms}
\label{sec:implementation}
\noindent We use the  implementation of the $hp$-FEM method described in detail 
in \cite{hrv1}. The strategy for computing the equipotential lines from 
the canonical domain onto the domain of interest can be found in \cite{hqr}.

The main difference between the two algorithms are the
cuts between the sets $E_j$, $j=1,2, \ldots, m$ in the case of type $R$, 
especially locating the saddle point between sets. We use the Ridge method, 
proposed by Ionova and Carter \cite{ionova-carter}, to locate the saddle points. 

To find the actual cutting curve, we bisect 
$\partial E_j$, $j=1, \ldots, m$ and move against the gradient of $u$. 
By doing so, we search for a point on $\partial E_j$ such that we end 
up within a tolerance from the saddle point.

If the cut can be computed analytically, the cut line can be embedded in the a pirori mesh and thus the
same mesh can be used in both problems. In this situation it is sufficient to perform
elemental integration once. The common blocks in the assembled linear systems 
can be eliminated as in \cite[{Section 4.2}]{hqr}.
In the general case, where the cutting has to be computed numerically the meshes may vary over large regions
and the positive bias from reusing the mesh is lost. In the numerical experiments below we have
used different refinements in two cases in order to test the sensitivity of the algorithm to mild
perturbations of the mesh.
The numerical algorithm is outlined in Figure~\ref{fig:alg}.
\begin{figure}
  \begin{algorithm}[H]
    \KwData{Domain $\symD$, tolerances $\epsilon_i$, $i=1,\ldots,4$ for the $\symM(\symD)$, saddle point, cuts, 
      and the reciprocal error, respectively.}
Discretize the domain $\symD$, solve $u_1$ and compute $\symM(\symD)$ with the desired tolerance $\epsilon_1$\;
\While{True}{
  Locate the saddle points (within tolerance $\epsilon_2$)\;
  Search the cuts (within tolerance $\epsilon_3$)\;
  Discretize the domain $\tilde{\symD}$, solve $u_2$, and compute $\symM(\tilde{\symD})$\;
  \lIf{The reciprocal error is below tolerance $\epsilon_4$ }{break}
  Decrease tolerances for the saddle points and the cuts, $\epsilon_2$ and $\epsilon_3$, respectively\;
}
\end{algorithm}
\caption{$R$-type: High-level description of the numerical algorithm.}\label{fig:alg}
\end{figure}

For the $Q$-type, similar iteration can be used to refine the potential
values. In this case it may be necessary to refine the geometric search
for the potential values.

\section{Numerical Experiments}
\label{sec:experiments}
In this section we discuss a series of benchmark problems and 
experiments carefully designed to illustrate different aspects of the algorithms.
In electrostatics the $Q$-type refers to resistor design problems
with multiple voltage domains
and the $R$-type to capacitor (electrical condenser) design ones. In practice, designing
integrated circuits multiple voltage domains is labor intensive and
there is a need for advanced design systems \cite{Iadanza}.
We have selected two problems of both types from literature and designed 
the experiments for $R$-type since
to our knowledge there are no reported benchmark problems with the actual maps
for the $R$-type domains. 

The use of the reciprocal relation as an error measure is formalized 
in the following definition:
\begin{definition}[Reciprocal error]
Using Proposition \ref{prop: reci} we can define two versions of the reciprocal error.
First for non-normalized jumps
  \begin{equation}\label{eq:nonnormal}
  e_r^d = |1 - \symM(\symD)/\symM(\tilde{\symD})|, 
\end{equation}
and second for the normalized ones
\begin{equation}\label{eq:normal}
  e_r^n = |1 - \symM(\symD)\symM(\tilde{\symD})|. 
\end{equation}
\end{definition}
and for convenience an associated error order
\begin{definition}[Error order]
  Given a reciprocal error $e_r^\star$, the positive integer $e_i$,
\begin{equation}
  e_i = |\lceil\log(e_r^\star)\rceil|,
\end{equation}
is referred to as the error order.
\end{definition}

Within the experiments we first consider cases with symmetries where the
cut can be computed analytically, and then a general case with two saddle points
(extraordinary points). 
We are interested in convergence in the energy norm as well as pointwise 
convergence.

For the general case the use of reciprocal error is not straightforward, 
however. The cuts must be approximated
numerically and the related approximation error leads to inevitable
\textit{consistency error} since the jumps depend on the chosen cuts.
Thus, in order to have a similar confidence in the general case as for the symmetric cases,
one should consider a sequence of approximations for the cuts as outlined above (Figure~\ref{fig:alg}). 
Here, however, we are
content to show via the conformal map that the chosen cut is a reasonable one, and
the resulting map has the desired characteristics.

Of course, the exact potential functions are not known. However, we can always
compute contour plots of the quantities of interest, that is, the absolute values of the derivatives,
and get a qualitative idea of the overall performance of the algorithm.
Naturally, this also measures the pointwise convergence of the Cauchy-Riemann 
problem.

Data on benchmarks and experiments, including representative numbers for degrees
of freedom assuming constant $p=12$, is given in Table~\ref{tbl:benchmark} and
Table~\ref{tbl:expCaps2}, respectively.
In all cases the setup of the geometry is the most expensive part in terms
of human effort and time. 
As is usual in $hp$-FEM, the computational cost in these
relatively small systems is in integration and handling of the sparse systems.
The actual computations take minutes on standard desktop hardware using our
implementation of the algorithms.

\subsection{Benchmarks}
In \cite{trefethen2} Trefethen gives an excellent introduction to
the connection between conformal maps and computation of resistances
of idealized planar resistors. In our setting the quantity of interest,
the resistance of the resistor, is equal to the modulus of the
conjugate domain.
\subsubsection{Computation of resistances for interior contacts}
Our first benchmark, Figure~\ref{fig:trefethen}, is a symmetric triply connected bar. 
On the interior square boundaries we have Dirichlet boundary conditions and
on the outer boundary Neumann ones. In the context of the application, 
the voltages are applied on the interior and the exterior is insulated.
This example was first discussed in \cite{trefethen2} where the computation 
were carried out with simply connected Schwarz-Christoffel transformations by 
exploiting the symmetry to subdivide the domain into four simply connected 
domains. In \cite{dek} the same problem is computed using the method of DeLillo et al. 
\cite{dep} without exploiting the symmetry.

The domain is enclosed by
$B = [0,3]\times[-1/2,1/2]$.
There are two square holes 
\[
H_1=[1/4,3/4]\times[-1/4,1/4],\quad H_1=[9/4,11/4]\times[-1/4,1/4],
\]
and indentations
\[
I_1=[1,2]\times[-1/2,-1/4],\quad I_2=[1,2]\times[1/4,1/2].
\]
The domain $\Omega = B \setminus (H_1 \cup H_2 \cup I_1 \cup I_2)$.
This problem is neither of type $R$ nor $Q$. 
Since the contacts are on the interior boundaries, cuts with Dirichlet
conditions must be present in the conjugate problem.
Here we cut along $y=0$,
set $u=0$, if $1 \leq x \leq 2$ and $u = \pm 1/2$ otherwise.

The computed value of resistance $R = 2.768867502692$ is equal
to those reported in \cite{trefethen2} and \cite{dek}. Notice that
in Figure~\ref{fig:trefethenc}, the maps include details also
around the contacts.
\begin{figure}
  \centering
  \subfloat[Domain.]{\includegraphics[width=0.3\textwidth]{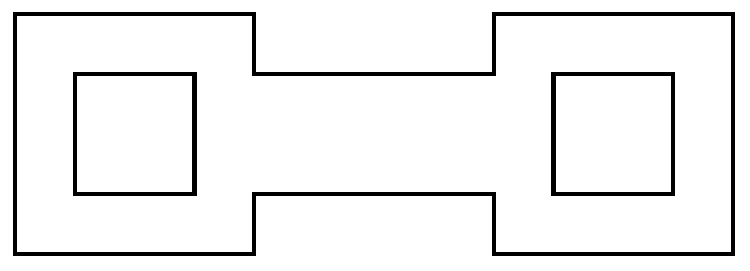}}\quad
  \subfloat[Mesh.]{\includegraphics[width=0.3\textwidth]{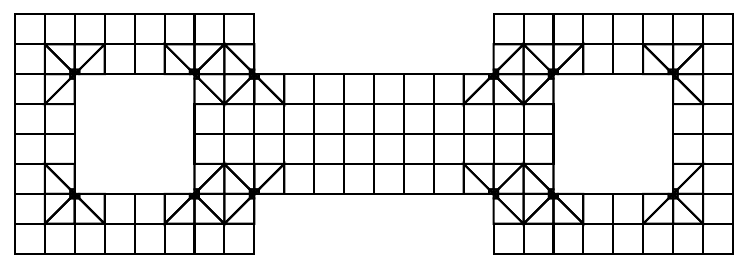}}\quad
  \subfloat[Map.]{\label{fig:trefethenc}\includegraphics[width=0.3\textwidth]{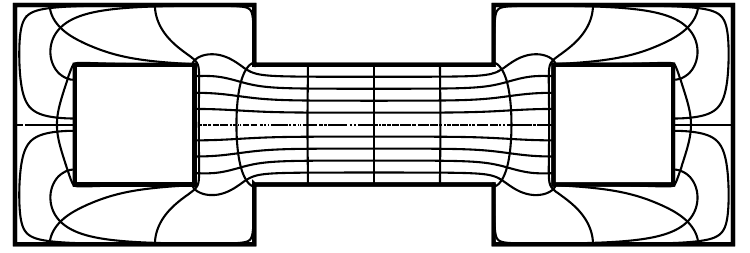}}
  \caption{Interior contacts.}\label{fig:trefethen}
\end{figure}

\subsubsection{Computation of resistances for quadrilaterals}
\label{sec:delillo}
Our second benchmark is of type $Q$ (see Figure~\ref{fig:delillo}), 
a resistor first computed  in \cite{dek}. The domain is enclosed by
$B = [-3/2,3/2]\times[-3/4,3/4]$.
There are two square holes (rotated by $\pi/4$)
\[
H_1=\{(-1/2,0),(-3/4,1/5),(-1,0),(-3/4,-1/5)\},
\]
\[H_2=\{(1/2,0),(3/4,-1/5),(1,0),(3/4,1/5)\},
\]
and indentations
\[
I_1=[-3/2,1/2]\times[-3/4,-1/2],\quad I_2=[-1/2,3/2]\times[1/2,3/4].
\]
The domain $\Omega = B \setminus (H_1 \cup H_2 \cup I_1 \cup I_2)$.
The contacts are on $x=-3/2$ and $y=-3/4$.
The computed value of resistance $R = 2.841998463680$ is equal
to that reported in \cite{dek}.

\begin{figure}
  \centering
  \subfloat[Domain.]{\includegraphics[width=0.3\textwidth]{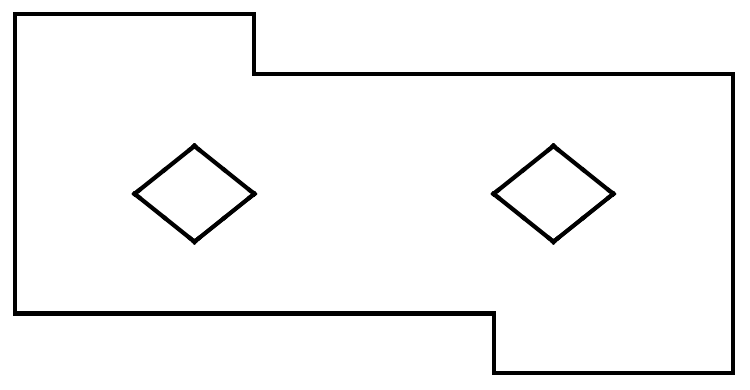}}\quad
  \subfloat[Mesh.]{\includegraphics[width=0.3\textwidth]{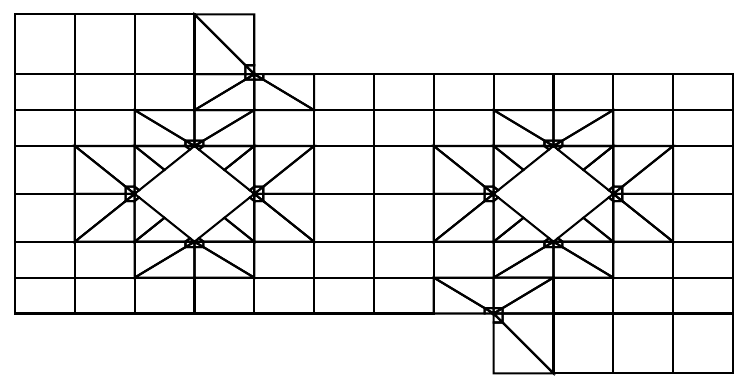}}\quad
  \subfloat[Map.]{\includegraphics[width=0.3\textwidth]{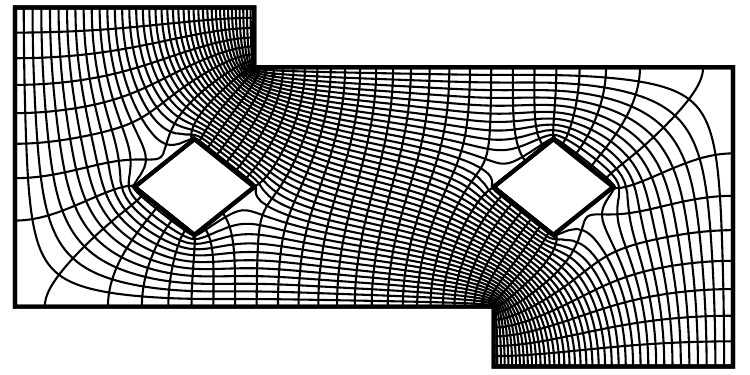}}
  \caption{Resistor.}\label{fig:delillo}
\end{figure}

\subsubsection{Computation of capacities}
\label{sec:bsv}
We consider two cases, Capacitor A and B, examples 7 and 10 from \cite{bsv},
respectively (see Figure~\ref{fig:bsv}). In both cases the domain $\Omega$ is
enclosed within $D=[-1,1]\times[-1,1]$. For Capacitor A, the plates are defined as the union of an equilateral triangle $T$ and its reflection in the real
axis. The vertices of $T$ are the points $(a,0)$, $(b,b-a)/\sqrt{3})$
and $(b,-(b-a)/\sqrt{3})$, where $0<a<b<1$. Here $a=1/5$ and $b=7/10$ and
the computed capacity $\textrm{cap} A = 9.49308124$ is within the estimated error
of the reference value. For Capacitor B, the plates are two slits $\overline{A_sB_s}$
and $\overline{C_sD_s}$,
defined by points $A_s=(-2/3,-1/2)$,
$B_s=(-2/3,1/2)$, $C_s=(1/2,-1/4)$, $D_s=(1/2,1/4)$.
The computed capacity $\textrm{cap} B = 8.47016014$ is also within the estimated error
of the reference value.

\begin{figure}
  \centering
  \subfloat[Capacitor A.]{\includegraphics[width=0.3\textwidth]{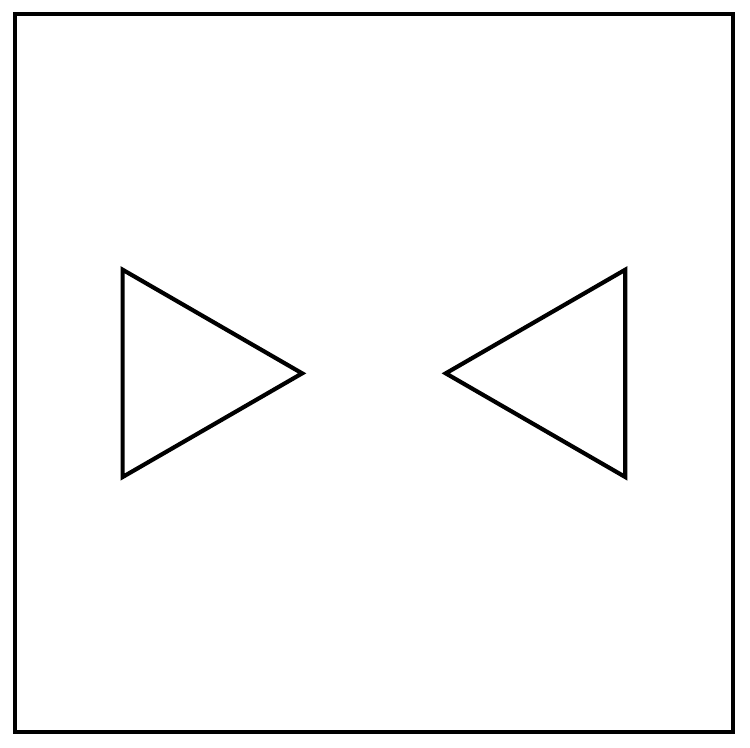}}\quad
  \subfloat[Capacitor B.]{\includegraphics[width=0.3\textwidth]{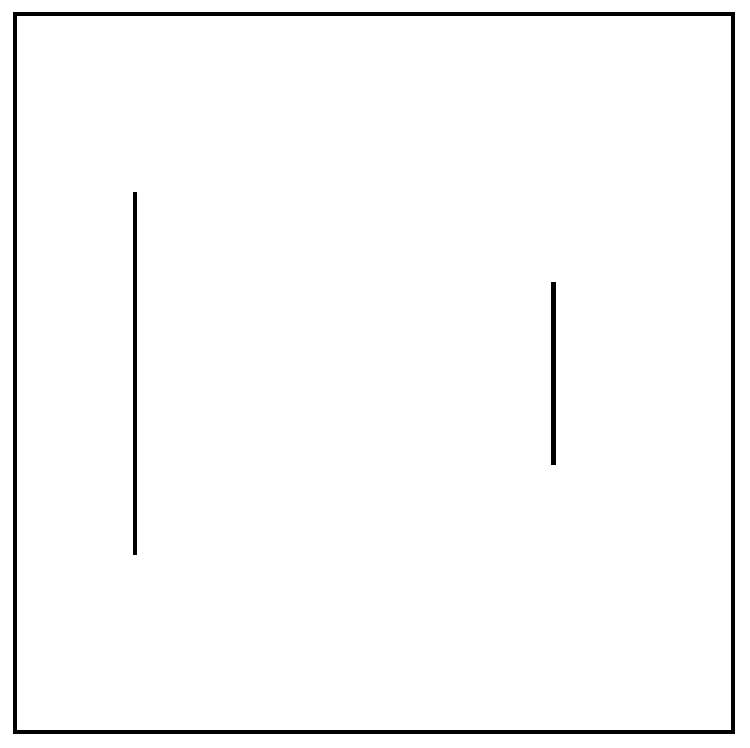}}
  \caption{Capacitors.}\label{fig:bsv}
\end{figure}

\begin{table}
  \centering
\subfloat[Computed resistance.]{\label{tbl:bexpCaps} \begin{tabular}{l|c|c|c|}
     Experiment & Capacity & Error order & (Reference) \\ \hline
     Interior contacts & 2.768867502692 & 12 & (2.76886750270)\\
     Resistor & 2.841998463680 & 11 & (2.8419984)\\
  \end{tabular}
}\\
\subfloat[Computed capacity. (Error) refers to the reported estimated error of the reference.]{\label{tbl:bsvCaps} \begin{tabular}{l|c|c|c|}
     Experiment & Capacity & (Reference) & (Error) \\ \hline
     Capacitor A & 9.49308124 & (9.4930811) & (4e-7)\\
     Capacitor B & 8.47016014 & (8.4701600) & (5e-7)\\
  \end{tabular}
}\\
\subfloat[FEM-data: Mesh: (nodes, edges, triangles, quads); Degrees of freedom given at $p=12$.]{\label{tbl:bexpCaps2} \begin{tabular}{l|c|c}
     Experiment & Mesh & DOF \\ \hline
     Interior contacts & (1569,2804,0,1280) & 187293\\
     Resistor & (667,1236,16,552) & 81935\\
     Capacitor A & (509,946,8,428) & 63143\\
     Capacitor B & (1013,1910,0,896) & 130439\\
  \end{tabular}
}
  \caption{Data on benchmarks.}\label{tbl:benchmark} 
\end{table}

\subsection{Reference Cases}
\label{refCases}

\begin{figure}
  \centering
  \subfloat[Domain.]{\label{fig:refsymmetricDomain}\includegraphics[width=0.30\textwidth]{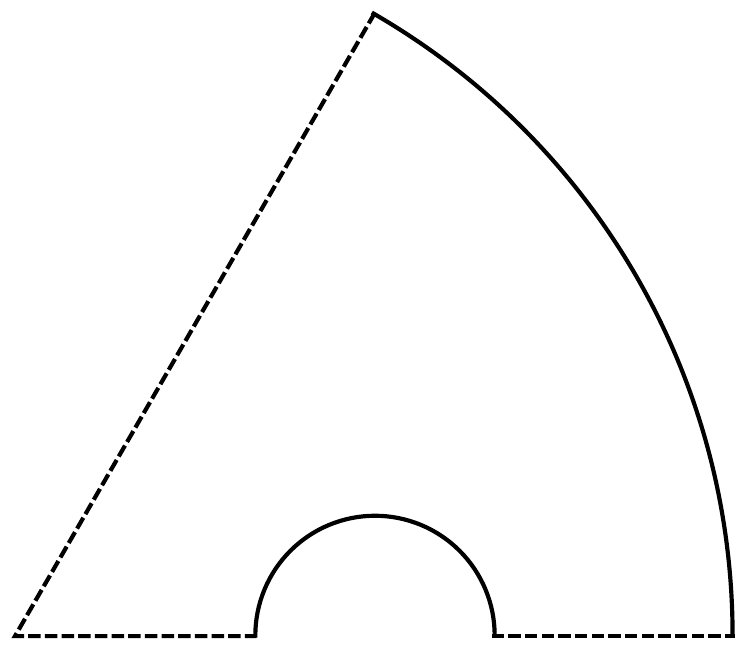}}\quad
  \subfloat[Mesh.]{\label{fig:refsymmetricMesh}\includegraphics[width=0.30\textwidth]{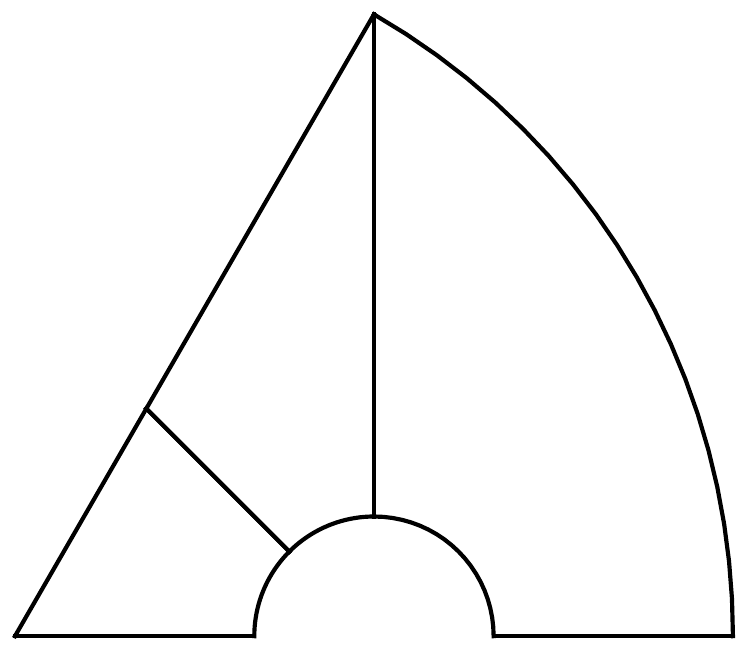}}\quad
  \subfloat[Map.]{\label{fig:refsymmetricMap}\includegraphics[width=0.30\textwidth]{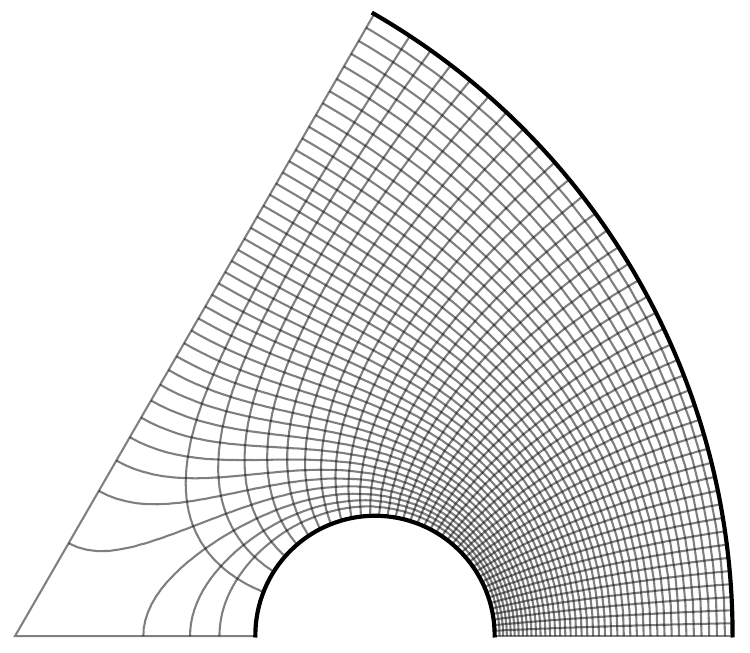}}
  \caption{Three Disks in Circle: Reference case 1.}\label{fig:refsymmetric}
\end{figure}

\begin{figure}
  \centering
  \subfloat[Domain.]{\label{fig:refccDomain}\includegraphics[width=0.30\textwidth]{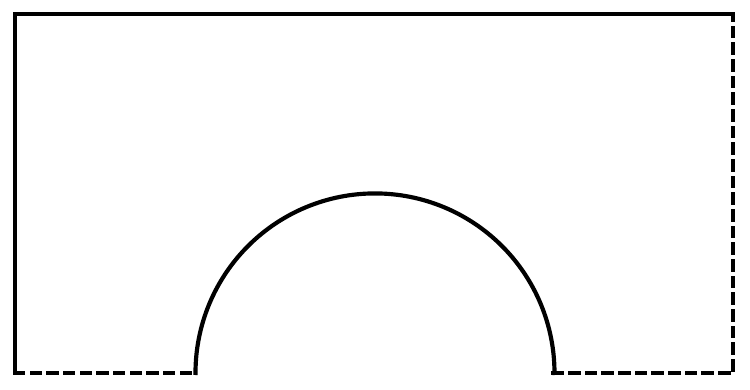}}\quad
  \subfloat[Mesh.]{\label{fig:refccMesh}\includegraphics[width=0.30\textwidth]{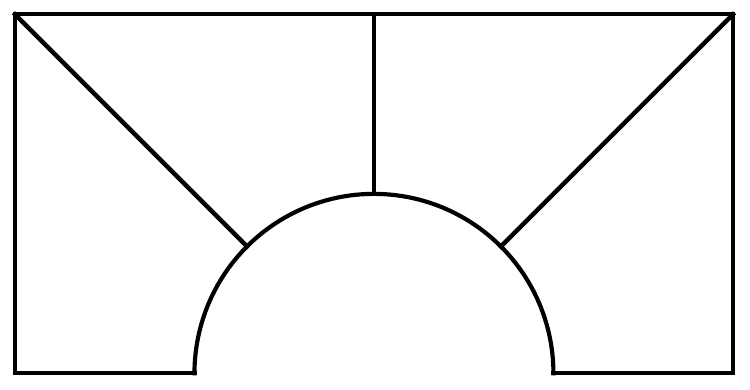}}\quad
  \subfloat[Map.]{\label{fig:refccMap}\includegraphics[width=0.30\textwidth]{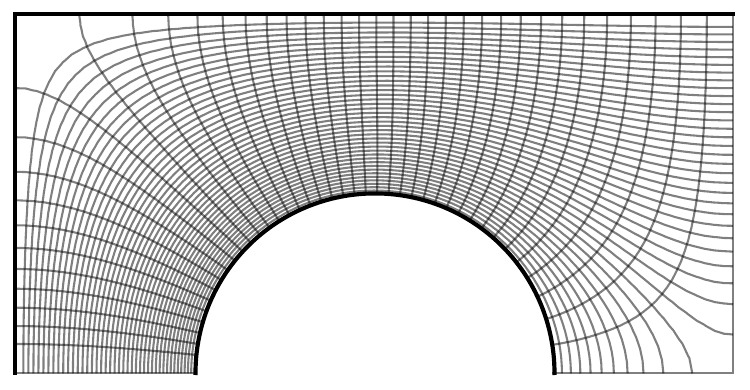}}
  \caption{Two Disks in Rectangle: Reference case 2.}\label{fig:refwcc}
\end{figure}

First we solve two standard problems up to very high accuracy in order to establish
reference results for capacities in cases with symmetries. The domains are shown in 
Figures~\ref{fig:refsymmetric} and \ref{fig:refwcc}. In both cases the $p$-version converges exponentially
as expected. The reference values for capacities are given in Table~\ref{tbl:refCaps}.

\begin{table}
  \centering
  \subfloat[Reference capacities.]{\label{tbl:refCaps} \begin{tabular}{l|c|c}
     Case & Capacity & Error order  \\ \hline     
     1 & 1.61245904853 & 12 \\
     2 & 3.48074407477 & 12 \\ \hline
  \end{tabular}
}\\
\subfloat[Computed capacities.]{\label{tbl:expCaps} \begin{tabular}{l|c|c}
     Experiment & Capacity & Error order \\ \hline
     Three Disks in Circle & 9.67475429123 & 12\\
     Two Disks in Rectangle & 13.922976299110 & 12\\
     Disk and Pacman in Rectangle & 13.3376294414 & 11\\
     Disk and Two Pacmen in Rectangle & 14.37(49228053) & 2 \\ \hline
  \end{tabular}
}\\
\subfloat[FEM-data: Mesh: (nodes, edges, triangles, quads); Degrees of freedom given at $p=12$.]{\label{tbl:expCaps2} \begin{tabular}{l|c|c}
     Experiment & Mesh & DOF \\ \hline
     Three Disks in Circle & (35, 52, 0, 18) & 2785\\
     Two Disks in Rectangle & (34,49,0,16) & 2509\\
     Disk and Pacman in Rectangle & (181,320,4,136) & 20377\\
     Disk and Two Pacmen in Rectangle &(353, 632, 8, 272) &40657 \\ \hline
  \end{tabular}
}
  \caption{$R$-type: Data on experiments.}\label{tbl:usedCapacities} 
\end{table}

\subsection{Symmetric Case: Three Disks in Circle}
\label{case1symmetric}
Consider a unit circle with three disks of radius $r = 1/6$ placed symmetrically so that
their origins lie on a circle of radius $r = 1/2$.
As indicated in Figure~\ref{fig:symmetricDomain} the cut can be computed analytically.
The blending function approach used to compute higher order curved elements is very accurate
if the element edges meet the curved edges at right angles. This is the reason for the
mesh of Figure~\ref{fig:symmetricMesh} where all edges adjacent to disks have been
set optimally.

Notice that due to symmetry, the scaled jumps could also be computed analytically.
In the numerical experiments only computed values of Table~\ref{tbl:expCaps} are used, however.
Since the cut is embedded in the mesh lines, both problems (the original and the conjugate)
can be solved using the same mesh. In this optimal configuration convergence 
in reciprocal relation
is exponential in $p$, which is a remarkable result, see Figure~\ref{fig:symmetricDomainReciprocal}.
Similarly, in Figure~\ref{fig:symmetricDomainContour}, it is clear that the derivatives 
have also converged over the whole domain.

\begin{figure}
  \centering
  \subfloat[Domain.]{\label{fig:symmetricDomain}\includegraphics[width=0.45\textwidth]{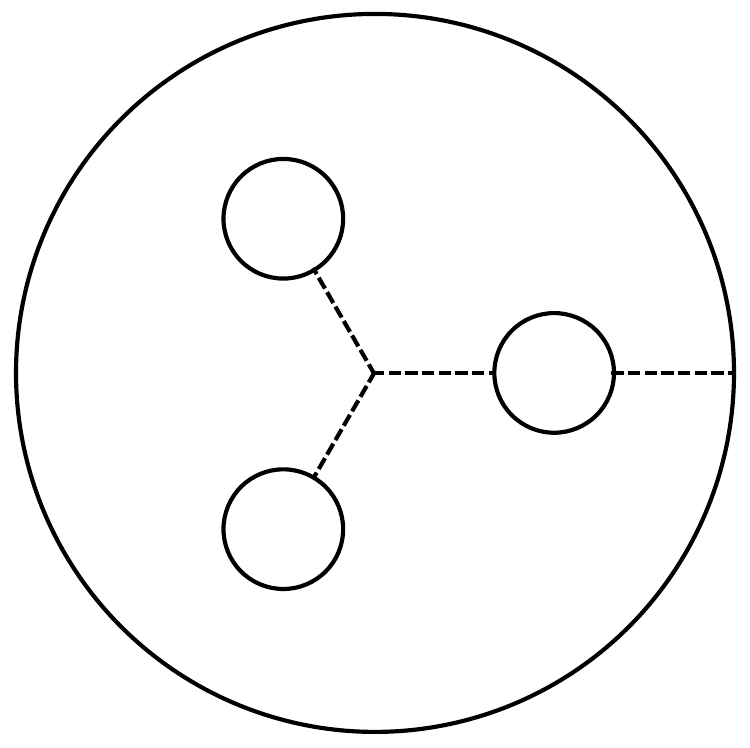}}\quad
  \subfloat[Mesh.]{\label{fig:symmetricMesh}\includegraphics[width=0.45\textwidth]{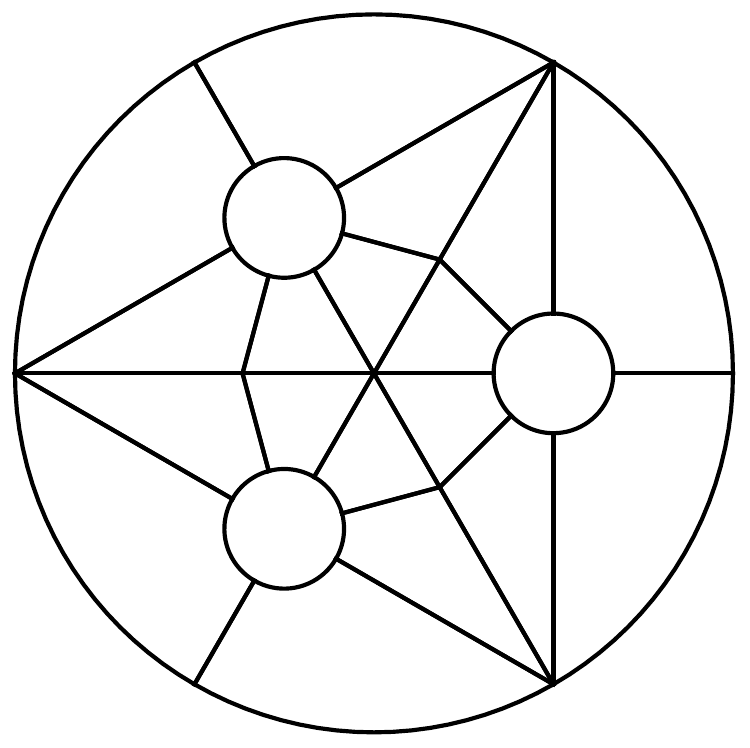}}\\
  \subfloat[Map.]{\label{fig:symmetricMap}\includegraphics[width=0.45\textwidth]{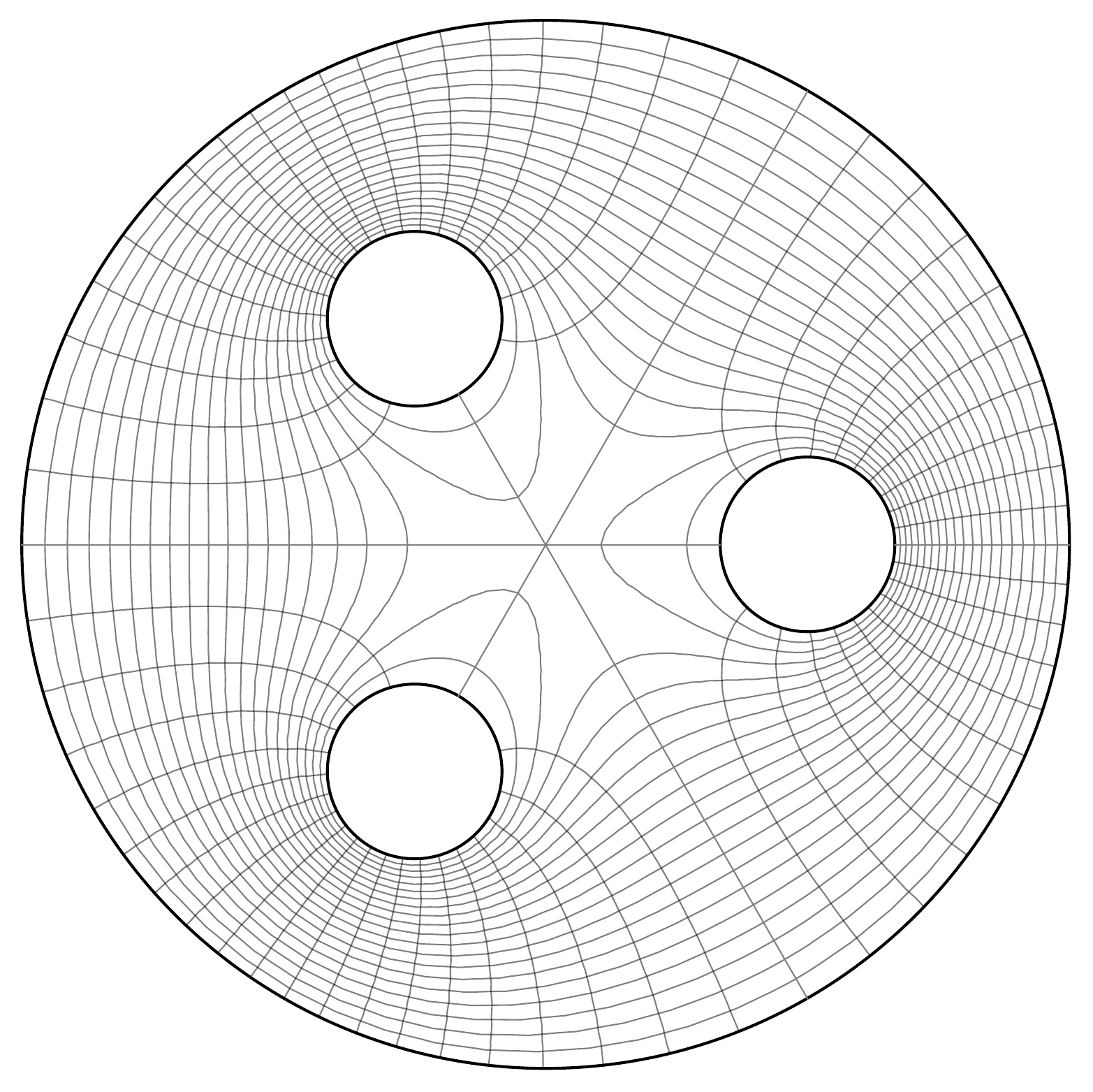}}
  \caption{$R$-type: Fully symmetric case.}\label{fig:symmetric}
\end{figure}

\begin{figure}
  \centering
  \subfloat[Cauchy-Riemann: Contour lines of $|\partial u/\partial x|$ and $|\partial v/\partial y|$.]{\label{fig:symmetricDomainContour}\includegraphics[width=0.45\textwidth]{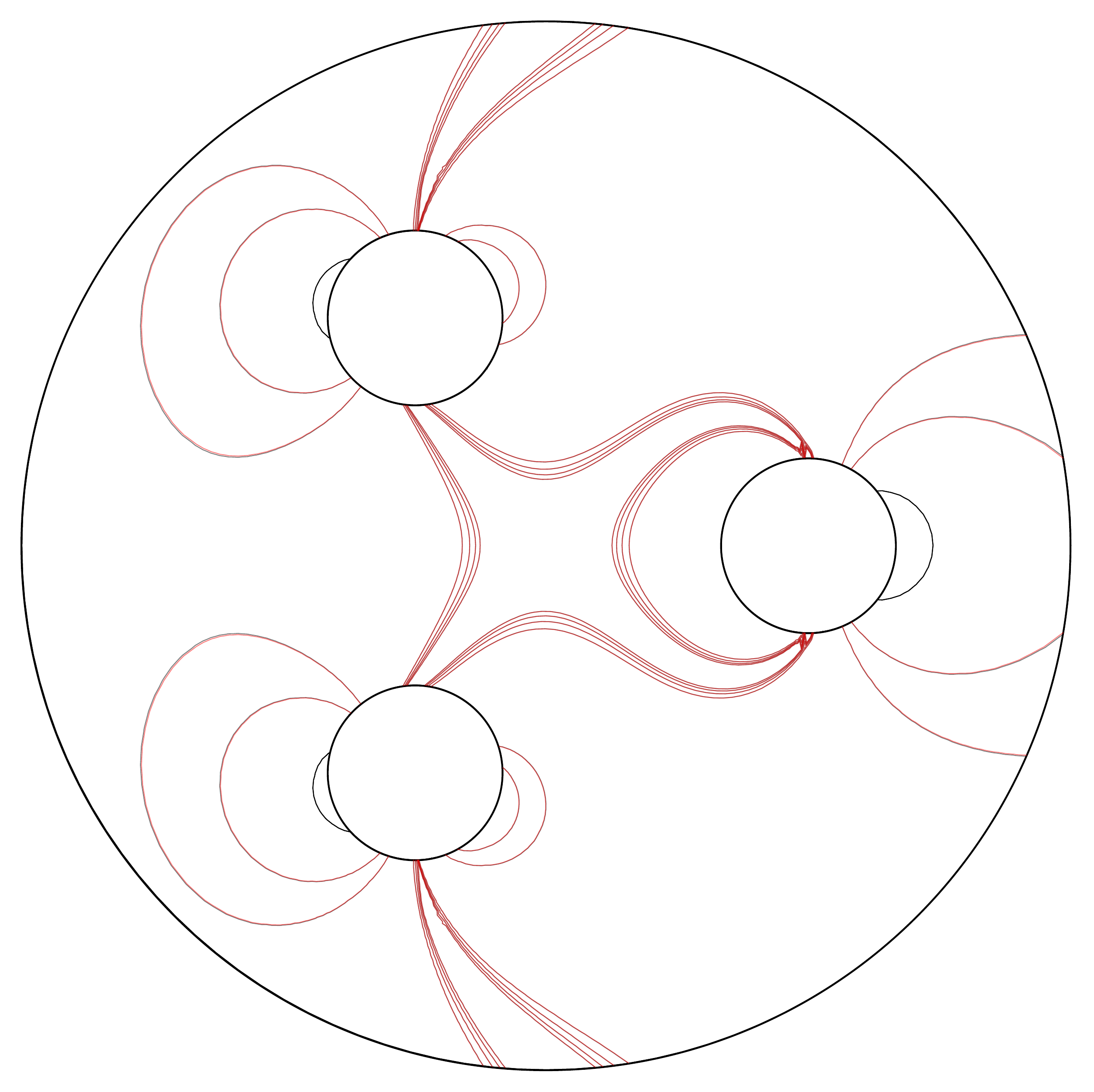}}\quad
  \subfloat[Reciprocal identity: Convergence in $p$; log-plot, error vs $p$.]{\label{fig:symmetricDomainReciprocal}\includegraphics[width=0.45\textwidth]{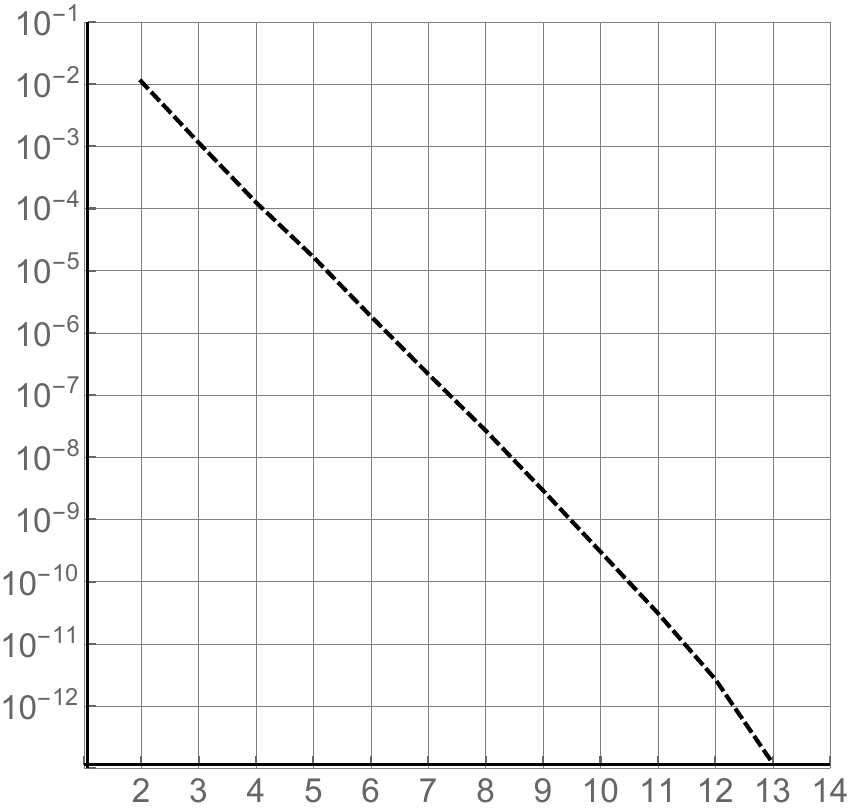}}
  \caption{$R$-type: Fully symmetric case.}\label{fig:symmetricCR}
\end{figure}

\subsection{Axisymmetric Cases}
\label{sec:axisymm}
In the next two cases we maintain axial symmetry and thus
analytic cuts.
In both cases the enclosing rectangle $R=[-1,3]\times[-1,1]$.
\subsubsection{Two Disks in Rectangle}
\label{sec:twodisks}
Consider two disks of radius $= 1/4$ with centres at $(0,0)$
and $(2,0)$, respectively.
The scaled jumps can be computed analytically, and standard jumps can be verified with the
reference case 2 in Table~\ref{tbl:refCaps}.
Once again, the reciprocal convergence in $p$ is exponential (see Figure~\ref{fig:symmetric2CRc}). 

\subsubsection{Disk and Pacman in Rectangle}
\label{sec:diskpacman}
Next the disk centred at $(2,0)$ is replaced by a disk with one quarter cut, the so-called pacman.
In this case we intentionally break the symmetry between meshes for the two problems.
The geometric refinement at the re-entrant corners is done in slightly different ways.
%Both approaches are shown in Figure~\ref{fig:general}.
The reciprocal convergence in $p$ is exponential, but with different rates at
lower and higher values of $p$. Also, the difference in the number of refinement levels
leads to mild consistency error which appears as loss of further convergence and accuracy at high $p$
(see Figure~\ref{fig:symmetric2CRd}). 

Here the jumps must be computed numerically (and tested against the computed capacity).
Jumps are with four decimals:
\[
  d_1 = 3.4808, d_2= 6.3761, d_3=3.4808.
\]
\begin{figure}
  \centering
  \subfloat[Two Disks in Rectangle.]{\includegraphics[width=0.45\textwidth]{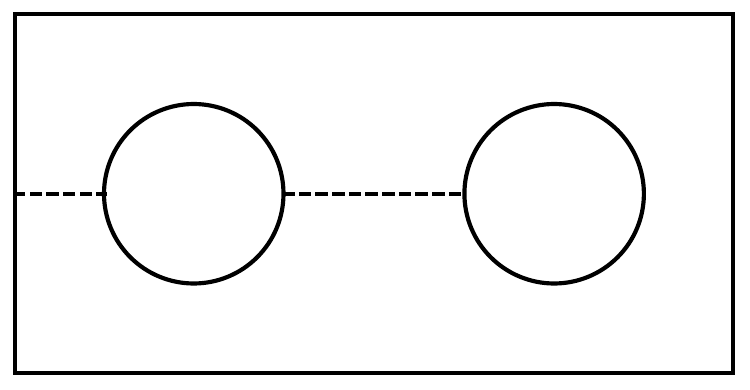}}\quad
  \subfloat[Disk and Pacman in Rectangle.]{\includegraphics[width=0.45\textwidth]{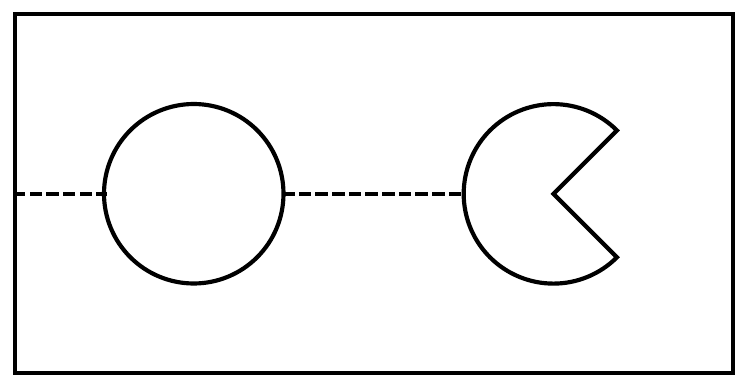}}\\
  \subfloat[Two Disks in Rectangle: Mesh.]{\includegraphics[width=0.45\textwidth]{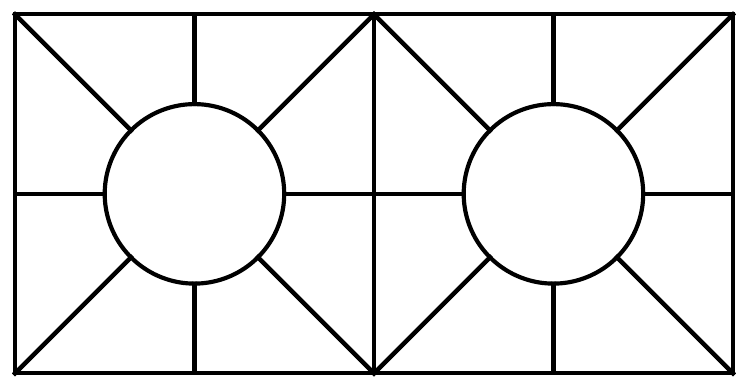}}\quad
  \subfloat[Disk and Pacman in Rectangle: Mesh.]{\includegraphics[width=0.45\textwidth]{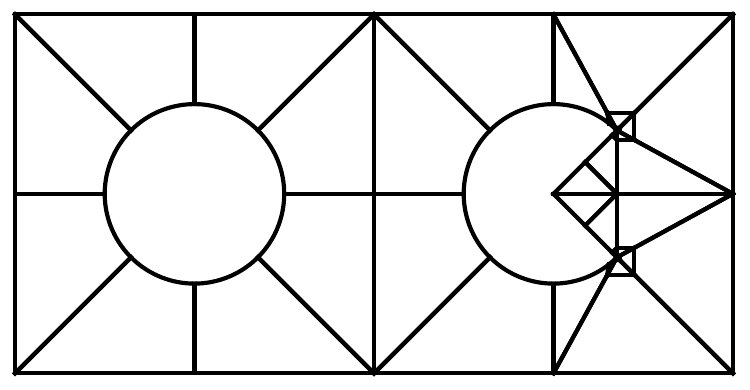}}\\
  \subfloat[Two Disks in Rectangle: Map.]{\includegraphics[width=0.45\textwidth]{CircleCircleMap}}\quad
  \subfloat[Disk and Pacman in Rectangle: Map.]{\includegraphics[width=0.45\textwidth]{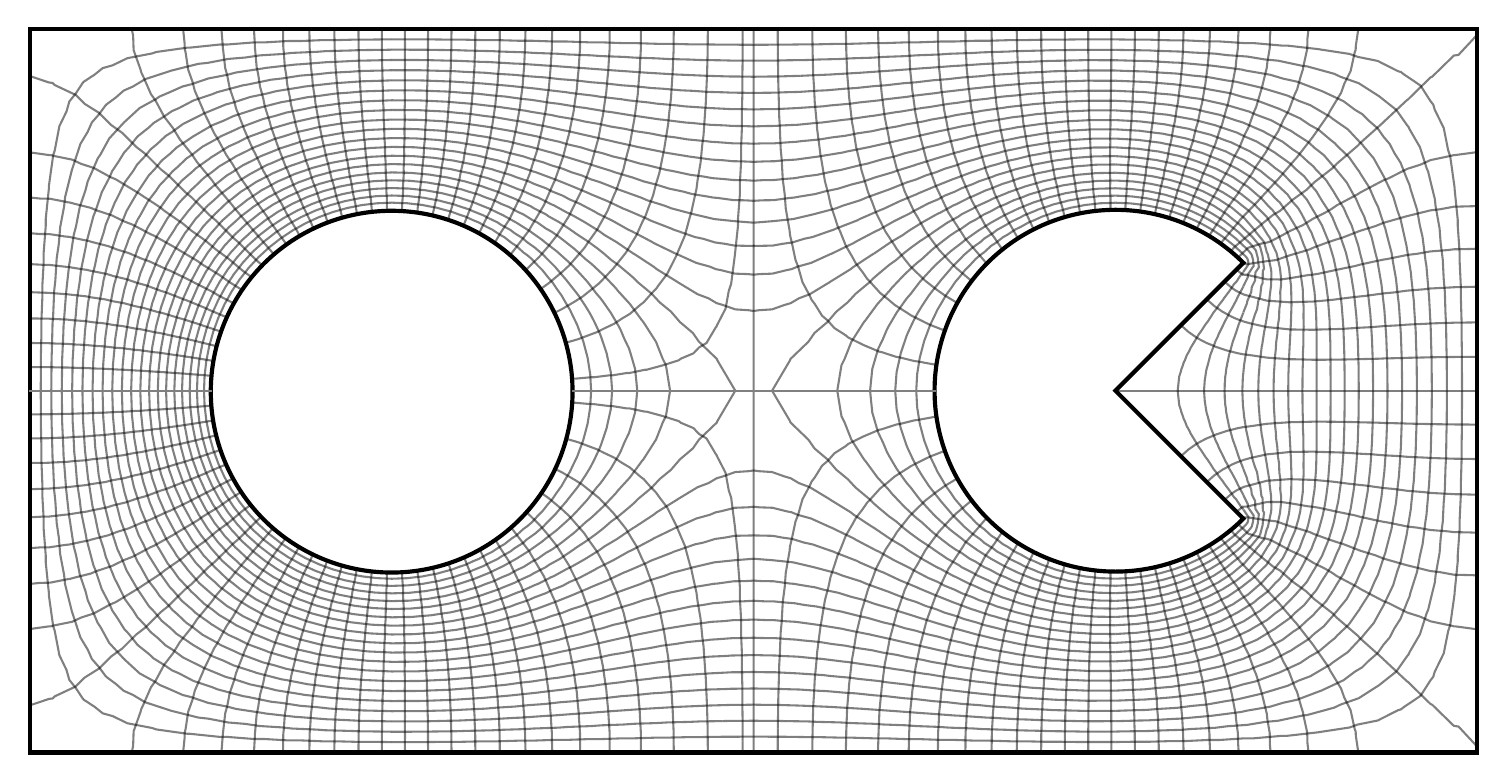}}
  \caption{$R$-type: Axially symmetric cases.}\label{fig:symmetric2}
\end{figure}

\begin{figure}
  \centering
  \subfloat[Cauchy-Riemann: Contour lines of $|\partial u/\partial x|$ and $|\partial v/\partial y|$.]{\includegraphics[width=0.45\textwidth]{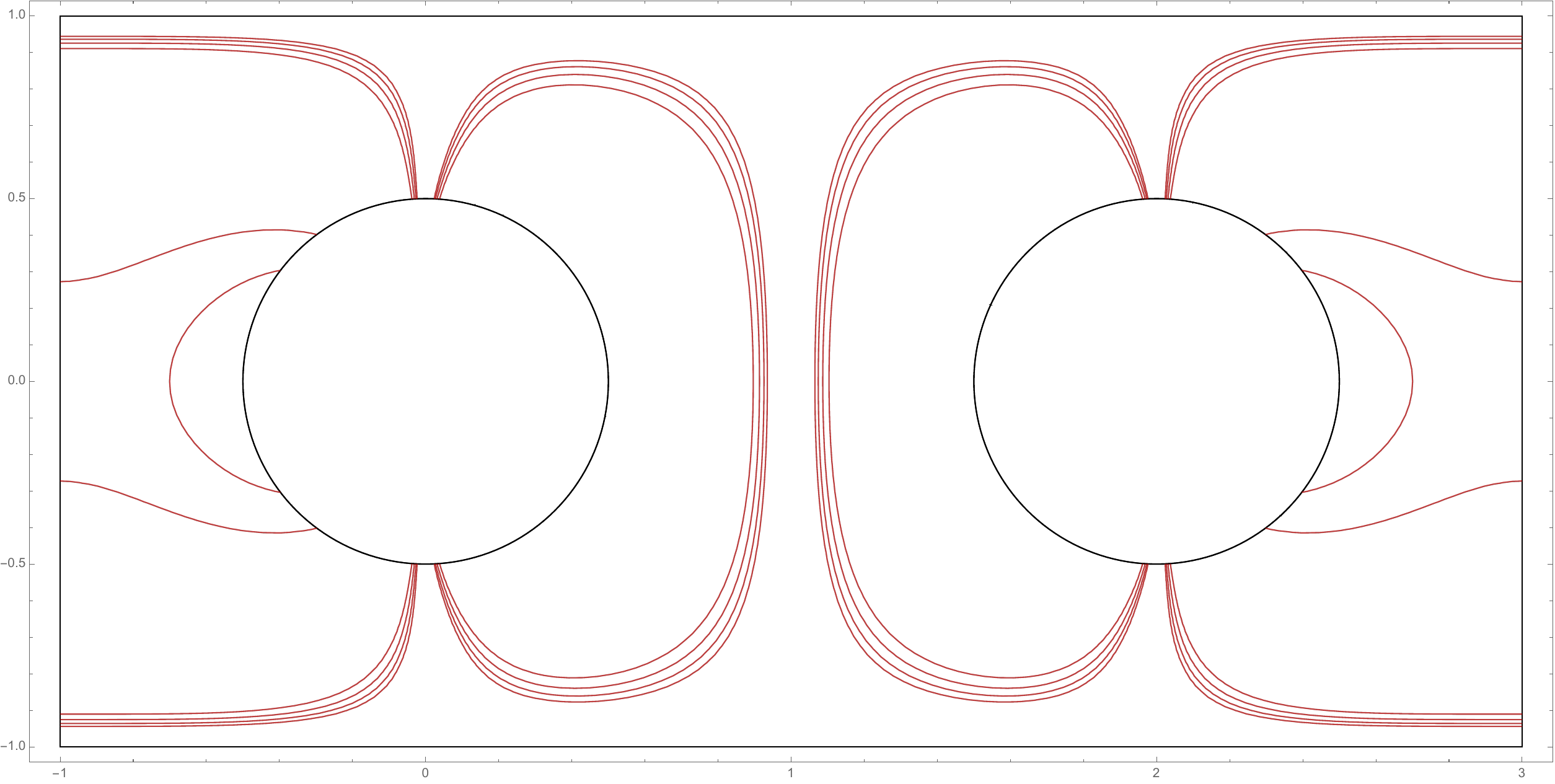}}\quad
  \subfloat[Cauchy-Riemann: Contour lines of $|\partial u/\partial x|$ and $|\partial v/\partial y|$.]{\includegraphics[width=0.45\textwidth]{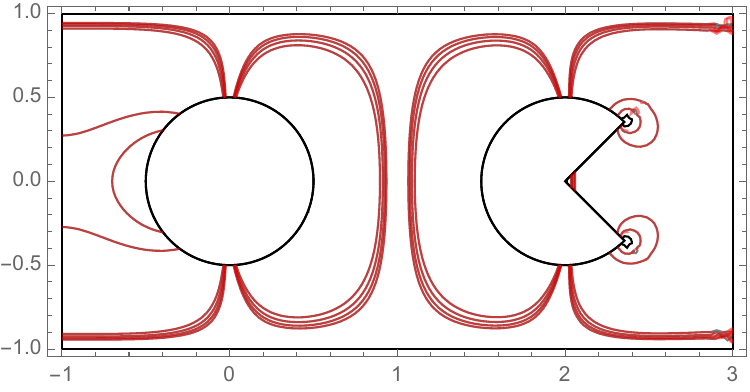}}\\
  \subfloat[Two Disks in Rectangle: Reciprocal identity: Convergence in $p$; log-plot, error vs $p$.]{\label{fig:symmetric2CRc}\includegraphics[width=0.45\textwidth]{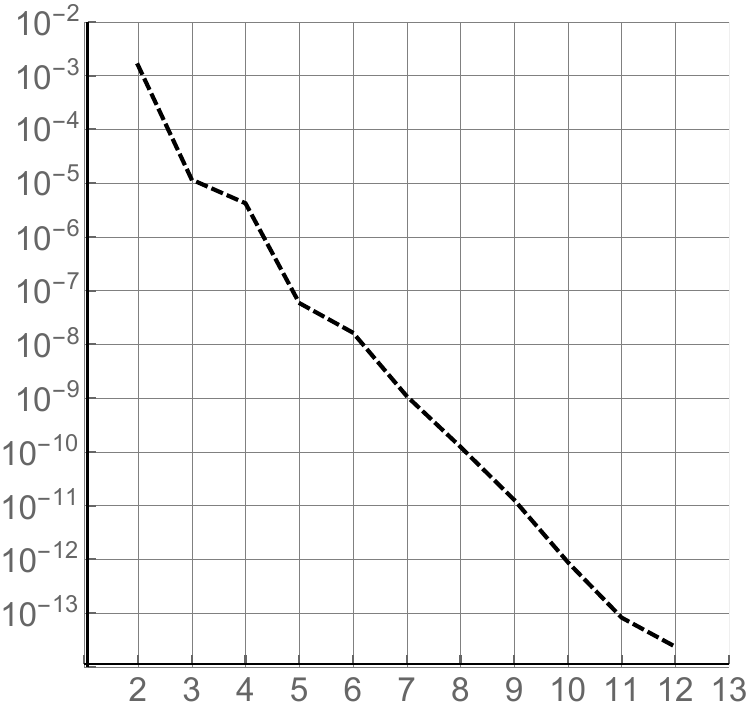}}\quad
  \subfloat[Disk and Pacman in Rectangle: Reciprocal identity: Convergence in $p$; log-plot, error vs $p$.]{\label{fig:symmetric2CRd}\includegraphics[width=0.45\textwidth]{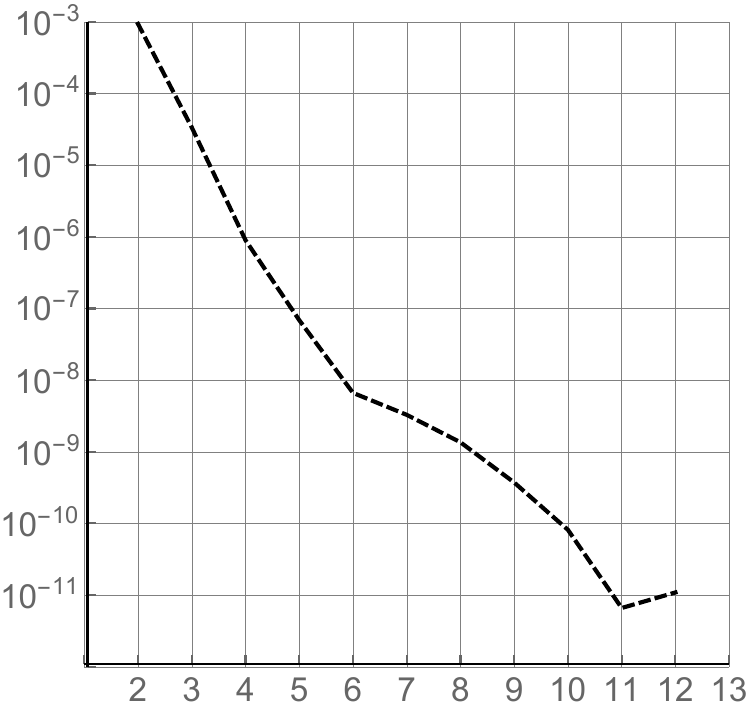}}\\
  \caption{$R$-type: Axially symmetric cases.}\label{fig:symmetric2CR}
\end{figure}

\section{Advanced Examples}

\subsection{Disk and Two Pacmen in Rectangle}
\label{sec:disktwopacmen}
The first example in this section is a general one with enclosing rectangle
$R=[-1,3]\times[-1,4]$ and one disk of radius $=1/4$ at $(0,1)$ and
two pacmen at $(2,0)$ and $(2,3)$. In this case the cuts cannot be determined analytically.
The effect of the cut in relation to the original problem can be seen by
comparing the meshes of Figures~\ref{fig:generala} and \ref{fig:generalb}.
One of the mesh points or nodes is moved to the saddle point and the
corresponding edge has been aligned with the cut.
As outlined before, the question of convergence in the reciprocal relation is
somewhat ambiguous in this case. The smallest error in the given configuration
is $0.00185$. The jumps in the derivatives across the cuts are also clearly
visible in Figure~\ref{fig:generald}. However, in Figure~\ref{fig:generalc}
we see how the contour lines do not cross the cuts except at the saddle points.
Jumps are with four decimals:
\[
  d_1=2.0001,  d_2=4.94015,  d_3=0.09500,
  d_4=5.1651,  d_5=2.1746.
\]
\begin{remark}
 In the Figure~\ref{fig:generald} the contour lines are given without any
 concern to the problem at hand. It would always be possible to, for instance,
 interpolate across the cuts and control the error. Here we have wanted to emphasize
 the effect of the approximate cut.
\end{remark}
\begin{figure}
  \centering
  \subfloat[Domain.]{\label{fig:generala}\includegraphics[width=0.45\textwidth]{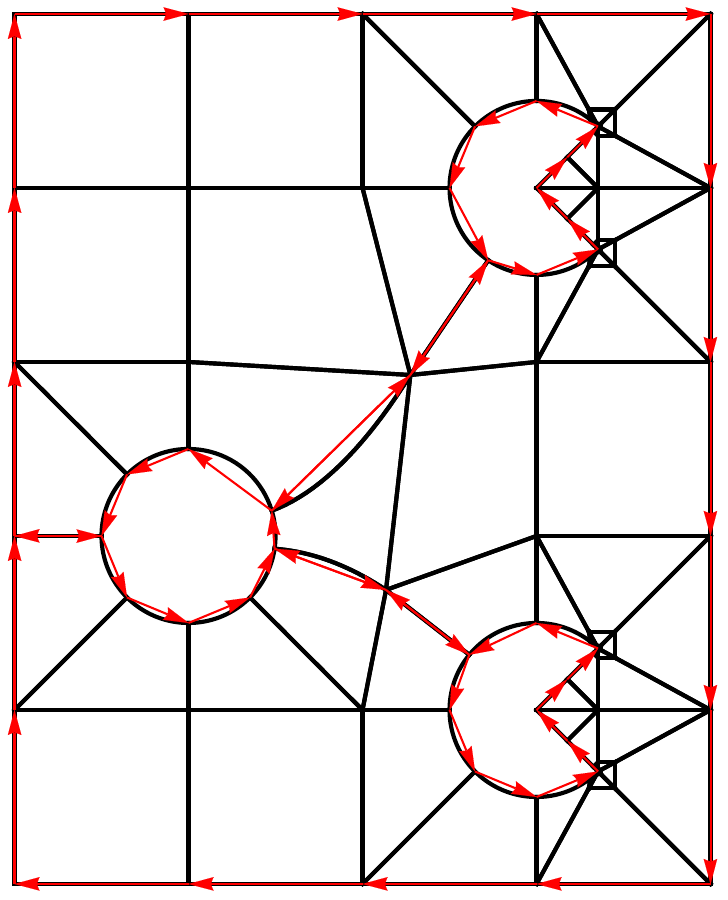}}\quad
  \subfloat[Mesh.]{\label{fig:generalb}\includegraphics[width=0.45\textwidth]{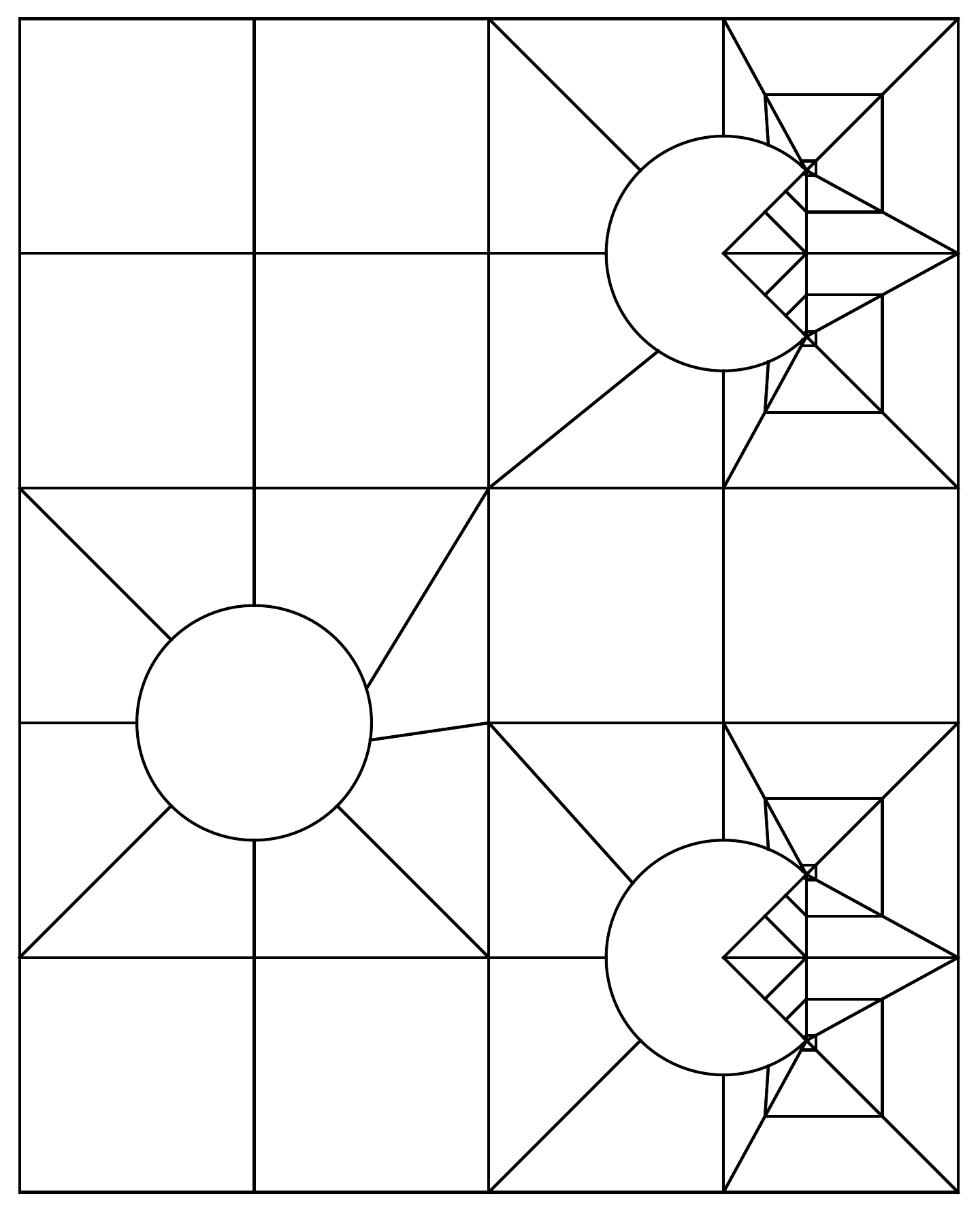}}\\
  \subfloat[Map.]{\label{fig:generalc}\includegraphics[width=0.45\textwidth]{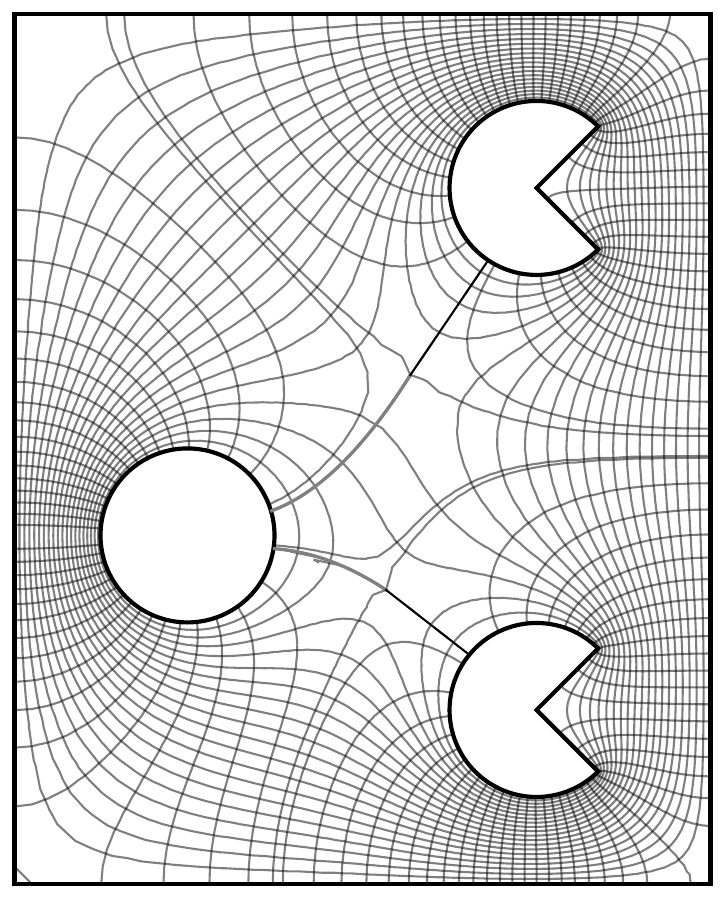}}\quad
  \subfloat[Cauchy-Riemann: Contour lines of $|\partial u/\partial x|$ and $|\partial v/\partial y|$.]{\label{fig:generald}\includegraphics[width=0.45\textwidth]{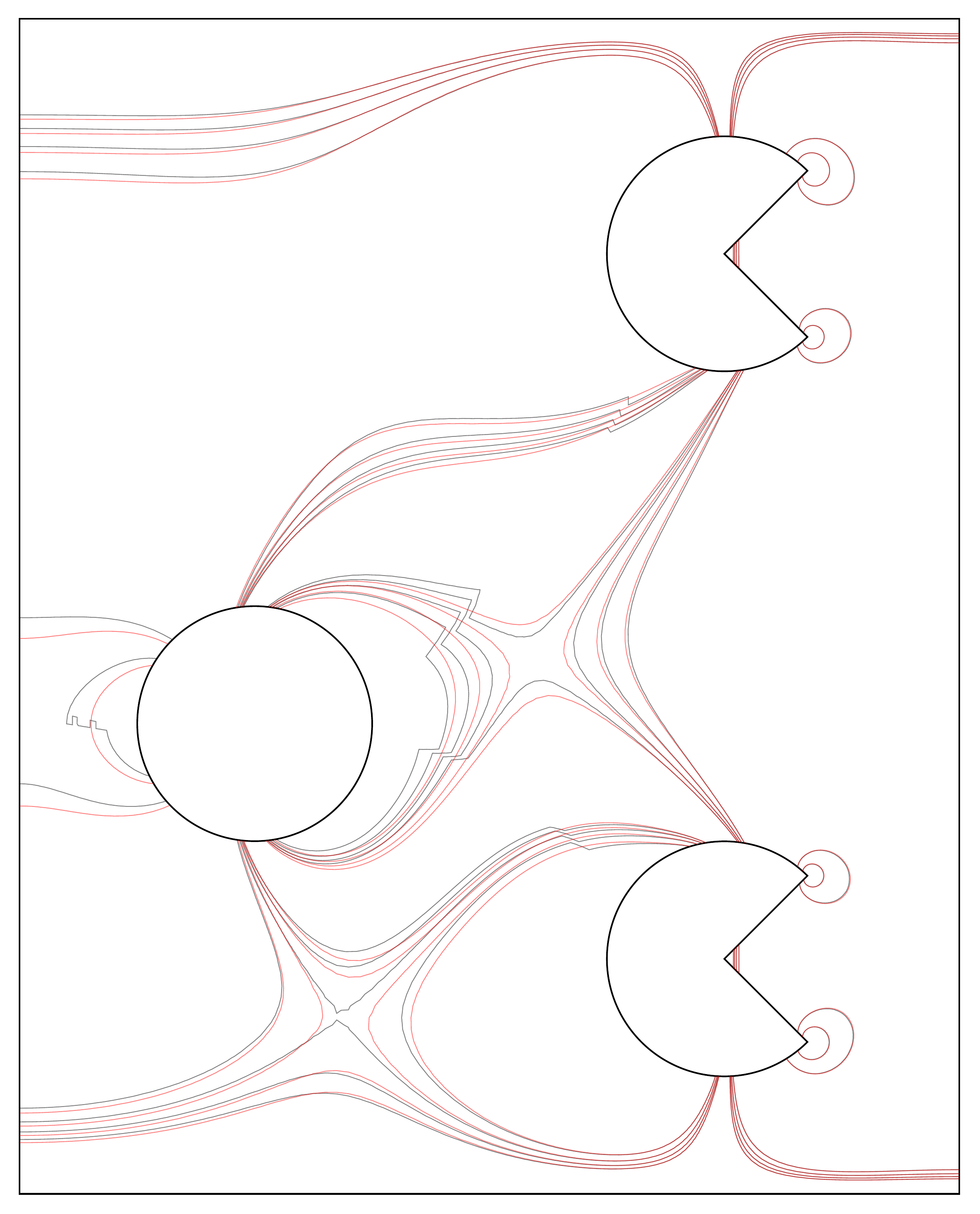}}
  \caption{$R$-type: Disk and Two Pacmen in Rectangle. Contour lines, in d, of derivatives are given without any concern to the problem at hand. Thus jumps across the cut are clearly visible. It is possible to interpolate across the cuts and control the error. Here we have left the jump to emphasize the effect of the approximate cut.}\label{fig:general}
\end{figure}

\subsection{Pacman and Droplet: Domain with Cusp}
As next example we consider an axisymmetric case with a pacman from above and a domain 
bounded by a Bezier curve:
\[
r(t) = \frac{1}{640} \left(45 t^6+75 t^4-525 t^2+469\right) +
\frac{15}{32} t
   \left(t^2-1\right)^2,\ t \in [-1,1].
\]
In \cite{hqr} a ring-domain with the same curve has been considered up to very high accuracy.
Notice that the ``droplet'' is designed so that also the tangents are aligned
for parameter values $t = \pm 1$, thus the opening angle is $2\pi$ requiring strong grading
of the mesh. The resulting maps are shown in Figure~\ref{fig:pacman-droplet}. 
\begin{figure}
  \centering
  \subfloat[$R$-type: Map.]{\label{fig:pacdropR}\includegraphics[width=0.45\textwidth]{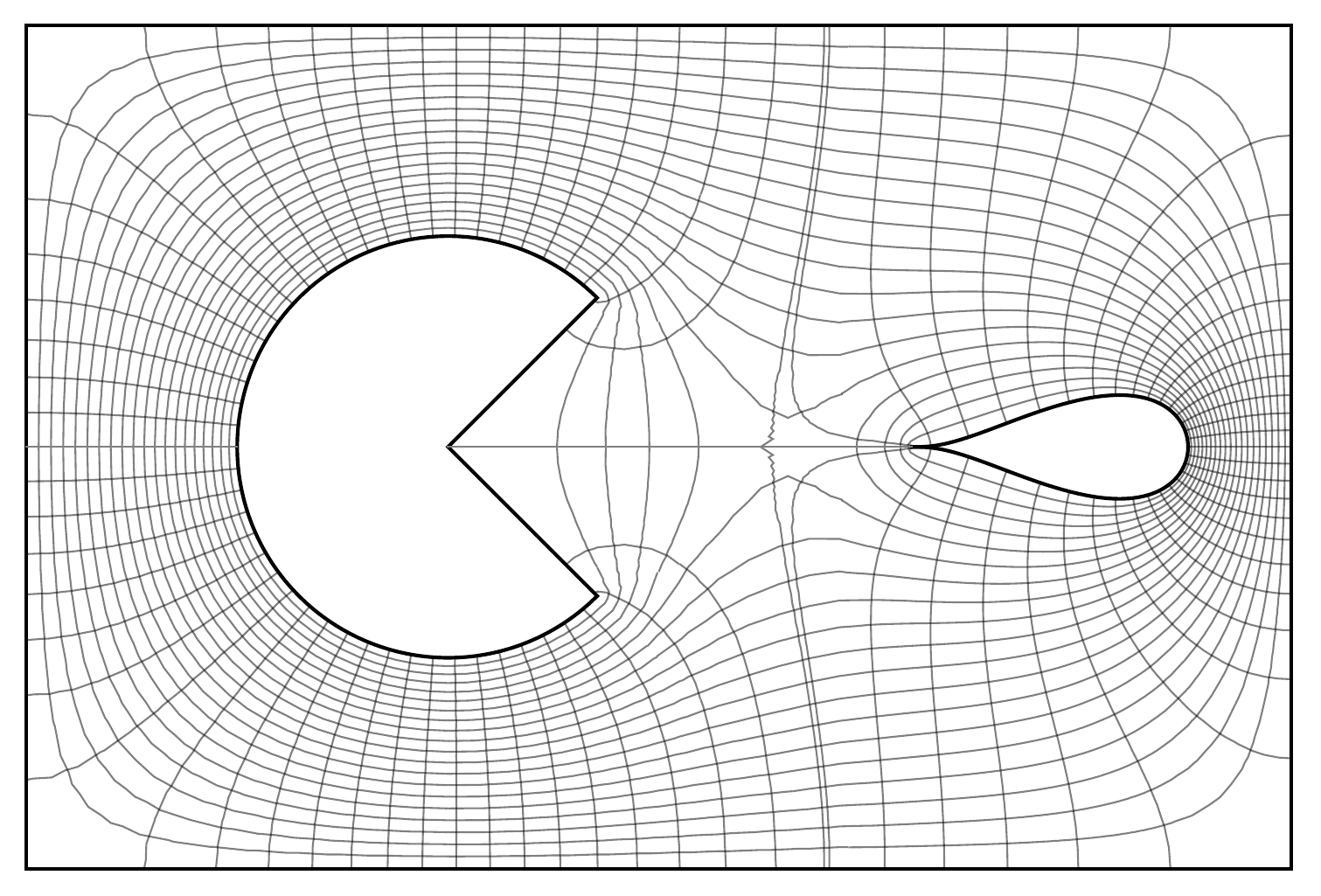}}\quad  
  \subfloat[$Q$-type: Map.]{\label{fig:pacdropQ}\includegraphics[width=0.45\textwidth]{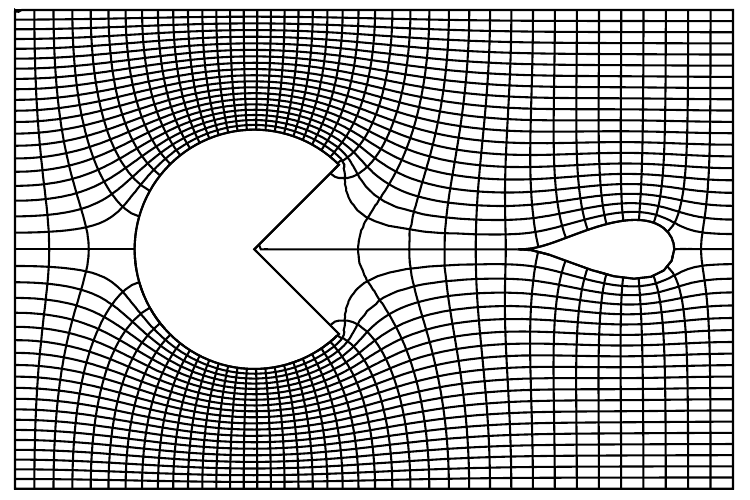}}  
  \caption{Pacman and Droplet.}\label{fig:pacman-droplet}
\end{figure}

\subsection{Perforated Domain: Domain with Uncertainty}
One fascinating and new application for conformal maps is book-keeping
of data in case of domains with uncertainty. 
Consider the perforated domain in Figure~\ref{fig:perforated}.
Let us assume that in plane elasticity we are interested in stresses under fixed loading. If the
manufacturing process leads to imperfections in the locations and sizes of the holes 
the task is to synthesize the stress fields over different realizations.

\begin{figure}
  \centering
  \subfloat[Nominal domain: Map.]{\label{fig:perforatedNominal}\includegraphics[width=0.45\textwidth]{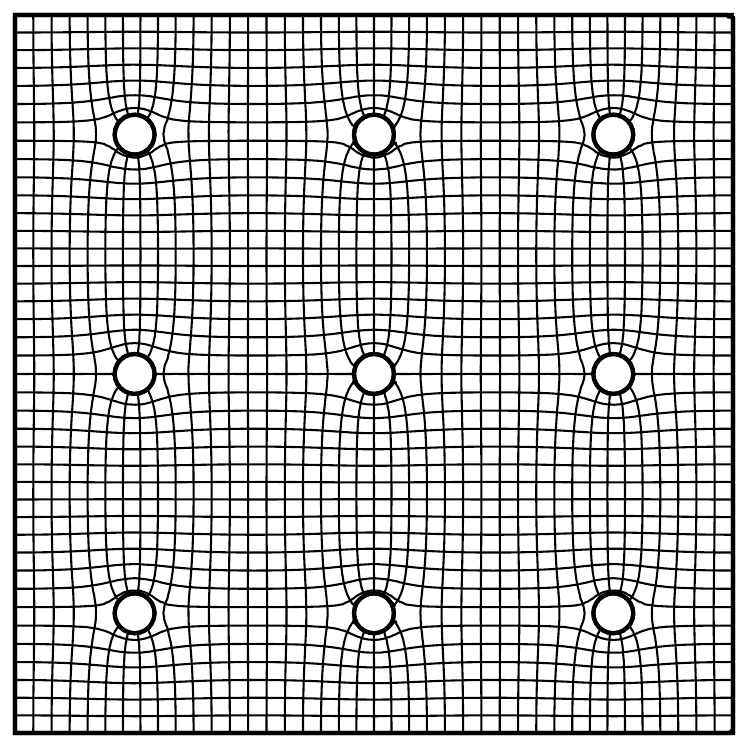}}\quad  
  \subfloat[Perturbed domain: Map.]{\label{fig:perturbedDomain}\includegraphics[width=0.45\textwidth]{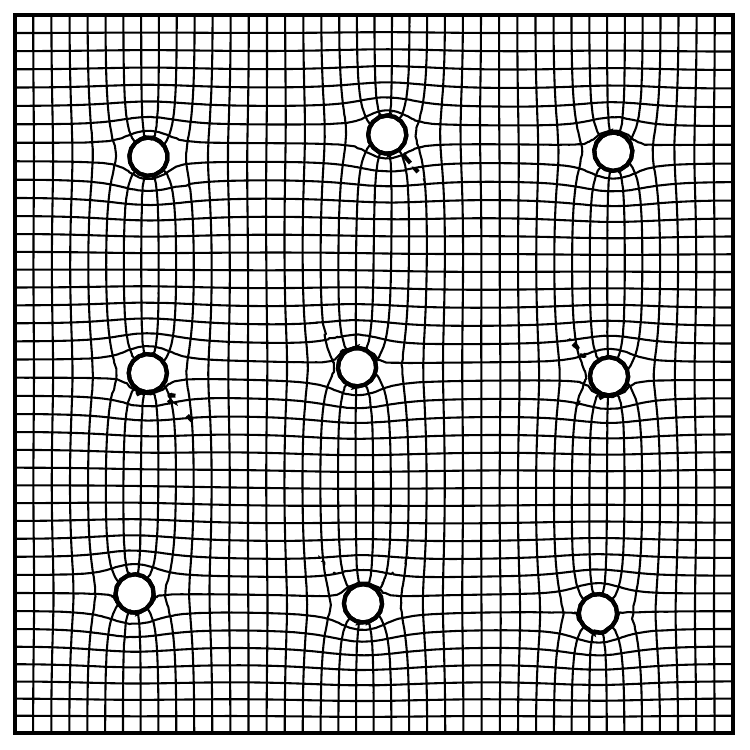}}  
  \caption{$Q$-type: Perforated System.}\label{fig:perforated}
\end{figure}
Let us refer to the domain without imperfections as the nominal domain 
(Figure~\ref{fig:perforatedNominal}).
The key observation is that once the canonical domain of the nominal domain
has been computed, the canonical domains of all realizations can further be mapped onto 
that of the nominal domain. As a result of this for every point of the nominal domain
a distribution of stresses is measured rather than a single value.
This approach has been succesfully applied in an industrial project
where a simply connected domain with uncertain boundary was studied \cite{jl}.

%\section{Discussion}
%The algorithm presented here has two straightforward applications after minor post-processing:
%quadrilateral mesh generation and cubature design. The saddle points (extraordinary points) are
%always nodes of cells with five nodes. These cells can always be further discretized by adding
%extra nodes. This means that the number of non-optimal cells in terms of orthogonality grows
%only linearly with the number of saddle points and is practically negligible. However, it is difficult
%to create cells of equal size, and in particular, all refinements around the saddle points
%must be compensated with strong grading elsewhere.
%Nevertheless, the accuracy requirements for domain discretization tend not to be very high, and we
%believe that there is great potential for both applications.
\begin{acknowledgement}
The authors wish to thank prof. R. Michael Porter for his careful reading of an earlier
version of this manuscript. The authors also wish to thank prof. T. DeLillo
and Dr. E. Kropf for help in setting up the example of Section~\ref{sec:delillo}.
\end{acknowledgement}

\end{document}